\documentclass[a4paper,11pt,final]{amsart}

\usepackage{graphicx}
\usepackage[super]{cite}
\usepackage{amssymb,amsmath}
\usepackage{mathrsfs}
\usepackage{latexsym}
\usepackage{prettyref}
\usepackage[colorlinks=true,linkcolor=black,citecolor=black]{hyperref}
\usepackage[T1]{fontenc}
\usepackage{dsfont}

\usepackage[a4paper,total={17cm,23cm}]{geometry}

\numberwithin{equation}{section}

\makeatletter
\def\mkmacro#1#2{\expandafter\xdef\csname dyn#1\endcsname{#2}}
\def\mylabel#1{\mkmacro{#1}{\@currenvir} \label{\@currenvir:#1}}
\def\myref#1{\prettyref{\csname dyn#1\endcsname :#1}}
\makeatother

\newrefformat{list}{(\hyperref[#1]{\ref{#1}})}

\newcounter{thmlist}
\renewcommand{\thethmlist}{\roman{thmlist}}

\newenvironment{thmlist}{%
\begin{list}{\textup{(\thethmlist)}}{\usecounter{thmlist}\setlength{\leftmargin}{\labelsep}}}%
{\end{list}}

\allowdisplaybreaks[2]

\newcommand{\N}{\ensuremath{\mathds{N}}} 

\newcommand{\R}{\ensuremath{\mathds{R}}} 

\newcommand{\C}{\ensuremath{\mathds{C}}} 

\newcommand\cA{\mathcal{A}}

\newcommand\cF{\mathcal{F}}

\newcommand\cH{\mathcal{H}}

\newcommand{\cU}{\mathcal{U}}

\newcommand{\frF}{\ensuremath{\mathfrak{F}}}
\newcommand{\frG}{\ensuremath{\mathfrak{G}}}

\newcommand{\ip}[2]{\ensuremath{\langle {#1},{#2} \rangle}} 
\newcommand{\ipbig}[2]{\ensuremath{\bigl\langle {#1},{#2} \bigr\rangle}} 

\newcommand{\abs}[1]{\ensuremath{\lvert {#1} \rvert}} 
\newcommand{\absbig}[1]{\ensuremath{\bigl\lvert {#1} \bigr\rvert}} 

\newcommand{\norm}[2]{\ensuremath{\lVert {#1} \vert {#2} \rVert}} 
\newcommand{\normw}[1]{\ensuremath{\lVert {#1} \rVert}} 

\newcommand{\integral}[4][]{\ensuremath{\int\limits^{{#1}}_{{#2}} {#3}\,\mathrm{d}{#4}}} 

\DeclareMathOperator{\sgn}{sign} 
\DeclareMathOperator*{\esssup}{ess\,sup} 
\DeclareMathOperator{\supp}{supp} 
\DeclareMathOperator*{\cartesian}{\times}

\theoremstyle{plain}
\newtheorem{theorem}{Theorem}[section] 
\newrefformat{theorem}{Theorem~\hyperref[#1]{\ref{#1}}}
\newtheorem{lemma}[theorem]{Lemma} 
\newrefformat{lemma}{Lemma~\hyperref[#1]{\ref{#1}}}
\newtheorem{proposition}[theorem]{Proposition} 
\newrefformat{proposition}{Proposition~\hyperref[#1]{\ref{#1}}}
\newrefformat{corollary}{Corollary~\hyperref[#1]{\ref{#1}}}

\theoremstyle{definition}
\newtheorem{definition}{Definition}[section] 
\newrefformat{definition}{Definition~\hyperref[#1]{\ref{#1}}}
\newrefformat{conjecture}{Conjecture~\hyperref[#1]{\ref{#1}}}
\newtheorem{example}{Example}[section] 
\newrefformat{example}{Example~\hyperref[#1]{\ref{#1}}}

\theoremstyle{remark}
\newtheorem{remark}{Remark} 
\newrefformat{remark}{Remark~\hyperref[#1]{\ref{#1}}}

\newcommand{\measureX}[2]{\ensuremath{\integral{\R^d}{\!\integral{\R^{d-1}}{{#1}}{s}}{t} + \integral{\R^d}{\!\integral{\R^{d-1}}{\!\int\limits^{1}_{-1} {#2}\,\frac{\mathrm{d}a}{\abs{a}^{d+1}}}{s}}{t}}} 
\newcommand{\leb}[1]{\ensuremath{\,\mathrm{d}{#1}}}
\newcommand{\condYb}{\ensuremath{\boldsymbol{(Y)}}}

\title{Inhomogeneous Shearlet Coorbit Spaces}

\author{Fabian Feise, Lukas Sawatzki}

\date{\today}



\begin{document}

\maketitle
\begin{abstract}
In this paper we establish inhomogeneous coorbit spaces related to the continuous shearlet transform and the weighted Lebesgue spaces $L_{p,v}, p\geq 1,$ for certain weights $v$.
We present an inhomogeneous shearlet frame for $L_2(\R^d)$ which gives rise to a reproducing kernel $R_\frF$ that is not contained in the space $\cA_{1,m_v}$.
To show that the inhomogeneous shearlet coorbit spaces are Banach spaces we introduce a generalization of the approach of Fornasier, Rauhut and Ullrich.
\end{abstract}



\section{Introduction}\label{sec:introduction}

When analyzing a given signal, the decomposition of the signal into a certain set of building blocks is crucial.
Which kinds of building blocks to choose depends on the information that one wants to extract from the signal.
Very popular kinds of building blocks are wavelets, especially when dealing with signals with isolated singularities.
Because of its isotropic nature, the wavelet transform cannot efficiently deal with anisotropic features, therefore several extensions of this framework were proposed, among those the shearlet transform.
While the wavelets consist only of dilated and translated copies of a mother function, the shearlets are also sheared in each scale, thereby changing the orientation of the functions.
This makes them especially well suited to deal with localized directional features in a signal.
Indeed, it was shown in~Ref.~\cite{ShearletsWavefrontSet,GrohsWavefrontSet} that the shearlet transform can be used to resolve the wavefront set of a signal and in~Ref.~\cite{ShearletsOptimallySparse} that the approximation of cartoon-like images with shearlets is optimally sparse.

Another main advantage of shearlets, which sets them apart from other such frameworks like the ridgelets~\cite{Ridgelets}, curvelets~\cite{Curvelets} or contourlets~\cite{Contourlets} for example, is, that the continuous shearlet transform, introduced and investigated in~Ref.~\cite{ShearletCoorbits,Dahlke:ShearletBanachFrames,Dahlke:ShearletArbitrary,Shearlets}, stems from the action of a square-integrable representation of a topological group, the so-called full shearlet group~$\mathds{S}$.
This property makes it possible to use the abstract coorbit theory, developed by Feichtinger and Gr\"ochenig in~Ref.~\cite{FG:Coorbit1,FG:Coorbit2,FG:Coorbit3}, to define smoothness spaces related to the shearlet transform by measuring the decay of the voice transform.
Shearlet coorbit spaces were investigated by Dahlke et al in a series of papers.~\cite{Dahlke:TracesEmbeddings,ShearletCoorbits,Dahlke:ShearletBanachFrames,Dahlke:ShearletArbitrary,Dahlke:CompactlySupp}
Since the shearlets being used to construct these spaces need to have vanishing moments, any polynomial part in a signal is ignored by the transform because for a polynomial $g$ one has $\mathcal{SH}(f+g)(x) = \ip{f+g}{\psi_x} = \ip{f}{\psi_x} = \mathcal{SH}f(x)$.
This leads to the resulting shearlet coorbit spaces being homogeneous spaces.
However, in practice the smoothness spaces being used, for example to analyze the regularity of the solution space of an operator equation, are usually inhomogeneous.
Therefore, inhomogeneous smoothness spaces related to the shearlet transform are also of interest.
In this paper we introduce non-homogeneous shearlet coorbit spaces by using a generalization of the coorbit theory developed by Fornasier, Rauhut, Ullrich et al.~\cite{FR:InhCoorbit,KempkaSchaeferUllrich:Coorbits,RauhutUllrich}
Their approach uses a more general parameter space for the transform, resulting in more design flexibility.
Instead of the parameter space being a locally compact topological group, it is only assumed to be a locally compact topological Hausdorff space, thereby allowing the construction of inhomogeneous coorbit spaces.
Moreover it is needed for the reproducing kernel $R_\frF$ to be integrabel, which poses difficulties in some applications.
For that reason we present a generalization of their approach in the sense that we only need $R_\frF$ to be integrabel for parameters $q>1$.

\subsection{Outline}
After giving a short overview of the main definitions and results of this generalized coorbit theory in~\prettyref{sec:coorbit_theory}, we use this approach in~\prettyref{sec:shearletcoorbit} to define a new shearlet transform given by a continuous frame $\frF = \{\psi_x\}_{x\in X}$ through the action \[ \mathcal{SH}_\frF f (x) = \ip{f}{\psi_x},\quad x \in X, \] where the frame is indexed by a topological Hausdorff space $X$ (without group structure).
We prove that an integrability condition for (integration) parameters $q>1$ on the kernel function \[ R_\frF: X\times X \rightarrow \C,\, (x,y) \mapsto \ip{\psi_y}{\psi_x} \] holds so that the coorbit spaces \[ \mathcal{SC}_{\frF,\tau,p}^{r} = \{ f\, \vert\, \mathcal{SH}_\frF f \in L_{p,v_{r}}(X) \},\quad p\geq 1, v_{r,n}\text{ weight function on }X, \] classifying distributions by the decay of their transform, are well-defined Banach spaces. As it turns out these spaces coincide for different $\tau$.
Furthermore we restrict ourselves to the case of odd dimensions.
This is due to the fact that otherwise our specific construction of the frame is not well-defined.

We also note that there are other approaches, not based on coorbit space theory, to develop inhomogeneous shearlet smoothness spaces.
In~Ref.~\cite{Labate:ShearletSmoothness} Labate, Mantovani and Negi used the notion of decomposition spaces to define shearlet smoothness spaces, while in~Ref.~\cite{Vera1,Vera2} Vera applied the framework of the $\varphi$-transform, introduced by Frazier and Jawerth, for this purpose.

\subsection{Notation}
We finish this section by stating a few notational conventions.
Throughout this paper $d \in \N$ with $d \geq 2$ is the space dimension.
We usually treat elements $x\in \R^d$ as $x = (x_1,\tilde{x})$ with $\tilde{x} = (x_2,\ldots,x_d) \in \R^{d-1}$.
For two elements $x,y\in \R^d$ we use the canonical inner product \[ x\cdot y = \sum_{i=1}^d x_i y_i. \]
The convention $\R^*$ is used for the set $\R\setminus\{0\}$, $\R_+$ will denote the set of all positive real numbers and $\R_{\geq 0}$ the set of all non-negative real numbers.

For a measure space $(X, \Sigma,\mu)$ with a weight function $v: X \rightarrow (0,\infty)$ we denote the usual (weighted) Lebesgue spaces by $L_{p,v}(X,\mu)$ or just by $L_{p,v}$, if the respective measure space is clear from the context, while $L_1^{\text{loc}}(X,\mu)$ is used for the space of locally integrable functions on $X$.
The norm for the weighted Lebesgue spaces is hereby given through $\normw{f}_{L_{p,v}} = \normw{f\cdot v}_{L_p}$.
For the unweighted Lebesgue spaces with $v \equiv 1$ we write $L_p(X,\mu)$ and $L_p$.
We use the Hilbert space $L_2(\R^d)$ of complex-valued, square-integrable functions on $\R^d$ with the inner product \[ \ip{f}{g}_{L_2(\R^d)} = \int_{\R^d} f(x) \overline{g(x)}\leb{x}. \]
For two functions $f,g\in L_2(\R^d)$ the convolution product $f\ast g$ is defined as \[ (f\ast g)(x) = \int_{\R^d} f(y)g(x-y)\leb{y}. \]
We write $\mathscr{C}^k, k \in \N_0$ for the space of functions $f:\R^d \rightarrow \C$, for which all (classical) partial derivatives $\partial^{\alpha}f$ for $\alpha \in \N_0^d, \abs{\alpha} \leq k$ exist and are continuous.
We also use $\mathscr{C}^\infty_0$ for the space of infinitely differentiable functions on $\R^d$ with compact support and $\mathscr{S}$ denotes the spaces of Schwartz-functions on $\R^d$. We will use the letter $q$ to refer to the kernel spaces $\cA_q$ and $\tau,\sigma$ to refer to the integrability parameters of the spaces of test functions $\cH_\tau$.
We denote with $p'=\frac{p}{p-1}$ the H\"older-dual of $p\geq 1$.

Concerning the Fourier transform of a function $f\in L_1(\R^d)$ we write $\hat{f} = \cF(f)$ using the convention \[ \cF(f)(\omega) := \int_{\R^d} f(x) e^{-2\pi i \omega\cdot x}\leb{x},\qquad \omega \in \R^d, \] with the same symbol being used for the extension to functions $f\in L_2(\R^d)$.

Given a measure space $(X,\Sigma,\mu)$ we say that a Banach space $Y$ of locally integrable, complex-valued functions on $X$ satisfies Condition \condYb, if it is solid, i.e. if from $f\in L_1^{\text{loc}}(X,\mu), g \in Y$ with $\abs{f}\leq \abs{g}$ almost everywhere it follows that $f\in Y$ with $\norm{f}{Y} \leq \norm{g}{Y}$.
Lastly, for quantities $a$ and $b$ we write $a \lesssim b$ if there exists a finite constant $C > 0$ so that $a \leq C\cdot b$, with the constant being independent of the relevant parameters.

\section{Generalized coorbit theory}\label{sec:coorbit_theory}

In this section we give a short overview of the generalized coorbit theory. We follow~Ref.~\cite{FR:InhCoorbit,RauhutUllrich} in our exposition. For our setting we introduce a generalization of their approach with respect to an additional integrability parameter.

To generalize the classical coorbit theory---which assumes a locally compact group as the underlying parameter space of the respective transform---the generalization of Fornasier and Rauhut allows for the parameter space to be of a more general nature.
In this case the parameter space $X$ is only assumed to be a locally compact Hausdorff space equipped with a positive Radon measure $\mu$.
In the following $\cH$ denotes a separable Hilbert space (the signal space), which is usually $L_2$, and $v$ is a weight function on $X$ while $Y$ is a Banach space of equivalence classes of almost everywhere equal, complex-valued functions on $X$.
We start with a set of functions $\frF = \{ \psi_x\}_{x\in X} \subset \cH$, which is indexed by the parameter space, and constitutes a tight continuous frame.
I.e., the map $X\rightarrow \C, x\mapsto \ip{f}{\psi_x}$ is measurable for each $f\in \cH$ and there exists a finite constant $A > 0$ such that
\begin{equation}\label{eq:frame}
A \norm{f}{\cH}^2 = \int_X \abs{\ip{f}{\psi_x}}^2\,\mathrm{d}\mu(x)\text{ for all }f\in \cH.
\end{equation}
Based on $\frF$, a signal transform on the space $\cH$ is introduced in the following way.
\begin{definition}
Let $\frF = \{ \psi_x\}_{x\in X} \subset \cH$ be a tight continuous frame.
Then the associated \emph{voice transform} is defined as the mapping \[ V_\frF : \cH \rightarrow L_2(X,\mu),\quad f \mapsto V_\frF f \] with \[ V_\frF f: X \rightarrow \C,\quad x \mapsto \ip{f}{\psi_x}. \]
\end{definition}
The above transform is well defined due to~\prettyref{eq:frame}.

\subsection{Kernel spaces}

In order for the resulting smoothness spaces to be well defined, conditions on the voice transform $V_\frF$ and therefore conditions on $\frF$ are needed.
In this approach the kernel function
\begin{equation}\label{eq:kernel}
R_\frF: X\times X \rightarrow \C, (x,y)\mapsto R_\frF (x,y) := V_\frF\psi_y(x) = \ip{\psi_y}{\psi_x},
\end{equation}
the \emph{reproducing kernel}, is used.
To formulate certain conditions on this kernel function the following spaces, classifying kernel functions in terms of integrability, are used.
For $1\leq q\leq\infty$ let \[ \cA_q := \Bigl\{ K: X \times X \rightarrow \C: K\text{ is measurable}, \norm{K}{\cA_q} < \infty \Bigr\} \] with
\begin{align*}
\norm{K}{\cA_q} := \max \bigg\{ &\esssup_{x\in X} \left(\int_X \abs{K(x,y)}^q\,\mathrm{d}\mu(y)\right)^{1/q},\\
& \esssup_{y\in X} \left(\int_X \abs{K(x,y)}^q\,\mathrm{d}\mu(x)\right)^{1/q}\bigg\}
\end{align*}
and the usual adaptation for $q=\infty$.
Through a weight function $v\geq 1$ on $X$ a kernel weight function is defined via \[ m_v: X \times X \rightarrow (0,\infty), (x,y) \mapsto \max\biggl\{\frac{v(x)}{v(y)},\frac{v(y)}{v(x)}\biggr\}. \]
Now the associated weighted kernel space $\cA_{q,{m_v}}$ is given by \[ \cA_{q,{m_v}} := \Bigl\{ K: X \times X \rightarrow \C: K\cdot m_v \in \cA_q\Bigr\} \] where \[ \norm{K}{\cA_{q,{m_v}}} := \norm{K\cdot m_v}{\cA_q}. \]
In the following, depending on the context, $K$ will also denote the kernel operator induced by the kernel function acting on a function $F$ through \[ K(F)(x) := \int_X K(x,y) F(y)\,\mathrm{d}\mu(y)\text{ for } x \in X. \]
This way a reproducing identity is established through the action of $R_\frF$, namely $R_\frF(V_\frF f)=V_\frF f$ for all $f\in\cH$.
The following auxiliary Lemma for kernel operators underlines the importance of the kernel spaces $\cA_{q,{m_v}}$.

\begin{lemma}\label{lemma:kernel_property}
Let $K$ be a kernel with $K\in\mathcal{A}_{q,m_v}$ for all $q>1$. Then we have the continuous embeddings
\begin{align*}
K(L_{p,v}(X,\mu)) \hookrightarrow L_{r,v}(X,\mu)
\end{align*}
for all $1< p < r \leq\infty$.
\end{lemma}

\begin{proof}
For fixed $1< p < r <\infty$ and $g\in L_{p,v}(X,\mu)$ with $\norm{g}{L_{p,v}}\leq 1$ arbitrary one has
\begin{align*}
\norm{K(g)}{L_{r,v}} &= \sup_{\substack{h\in L_{r',\frac{1}{v}} \\ \norm{h}{L_{r',\frac{1}{v}}}\leq 1}} |\ip{K(g)}{h}| \\
&\leq \sup_{\substack{h\in L_{r',\frac{1}{v}} \\ \norm{h}{L_{r',\frac{1}{v}}}\leq 1}} \int_X \int_X|K(x,y)g(y)h(x)|\,\mathrm{d}\mu(x)\,\mathrm{d}\mu(y) \\
&=: \sup_{\substack{h\in L_{r',\frac{1}{v}} \\ \norm{h}{L_{r',\frac{1}{v}}}\leq 1}} I_{K,p,r},
\end{align*}
where $r'$ denotes the H\"older-dual of $r$ satisfying $1/r+1/r'=1$. For some $0<\varepsilon<1/p-1/r$ we set $\alpha:=r>0$, $\beta:=p'>0$, \mbox{$1/\gamma:=1/p-1/r>0$}, $a:=1/r+\varepsilon$, $b:=p/r$, $c:=1/r'-\varepsilon$, $d:=r'/p'$, $e:=1-p/r$, $f:=r'/p-r'/r$. These choices suffice the following relations:
\begin{align*}
1/\alpha+1/\beta+1/\gamma=1, \hspace{0.5cm} a+c&=1, \hspace{0.5cm} b\alpha=p, \hspace{0.5cm} d\beta=r', \hspace{0.5cm} a\alpha>1, \\
b+e&=1, \hspace{0.5cm} e\gamma=p, \hspace{0.5cm} f\gamma=r', \hspace{0.5cm} c\beta>1, \\
d+f&=1.
\end{align*}
By applying the three-way Young inequality, see~\prettyref{lemma:threeway_youngs_inequality}, we obtain
\begin{align*}
I_{K,p,r} &\leq \int_X\int_X|K(x,y)m_v(x,y)|^a|f(y)v(y)|^b\cdot|K(x,y)m_v(x,y)|^c|h(x)v(x)^{-1}|^d\\
&\hspace{2cm}\cdot|g(y)v(y)|^e|h(x)v(x)^{-1}|^f\,\mathrm{d}\mu(x)\,\mathrm{d}\mu(y) \\
&\leq \frac{1}{\alpha}\int_X\int_X|K(x,y)m_v(x,y)|^{a\alpha}|g(y)v(y)|^p\,\mathrm{d}\mu(x)\,\mathrm{d}\mu(y) \\
&\hspace{2cm}+ \frac{1}{\beta}\int_X\int_X|K(x,y)m_v(x,y)|^{c\beta}|h(x)v(x)^{-1}|^{r'}\,\mathrm{d}\mu(x)\,\mathrm{d}\mu(y) \\
&\hspace{2cm}+ \frac{1}{\gamma}\int_X\int_X|g(y)v(y)|^p|h(x)v(x)^{-1}|^{r'}\,\mathrm{d}\mu(x)\,\mathrm{d}\mu(y).
\end{align*}
For the first summand we deduce the estimation
\begin{align*}
&\int_X\int_X|K(x,y)m_v(x,y)|^{a\alpha}|g(y)v(y)|^p\,\mathrm{d}\mu(x)\,\mathrm{d}\mu(y) \\
&\hspace{2cm} \leq \left(\esssup_{y\in X}\int_X|K(x,y)|^{a\alpha}|m_v(x,y)|^{a\alpha}\,\mathrm{d}\mu(x)\right)\int_X|g(y)|^p|v(y)|^p\,\mathrm{d}\mu(y) \\
&\hspace{2cm} \leq \norm{K}{\mathcal{A}_{a\alpha,m_v}}^{a\alpha}\norm{g}{L_{p,v}}^p
\end{align*}
and the other two summands can be treated analogously. Thus we obtain
\begin{align*}
I_{K,p,r}&\leq \frac{1}{\alpha}\norm{K}{\mathcal{A}_{a\alpha,m_v}}^{a\alpha}\norm{g}{L_{p,v}}^p + \frac{1}{\beta}\norm{K}{\mathcal{A}_{c\beta,m_v}}^{c\beta}\norm{h}{L_{r',\frac{1}{v}}}^{r'}\\
&\quad +\frac{1}{\gamma}\norm{g}{L_{p,v}}^p\norm{h}{L_{r',\frac{1}{v}}}^{r'} \\
&\leq \max\left\{1,\norm{K}{\mathcal{A}_{a\alpha,m_v}}^{a\alpha},\norm{K}{\mathcal{A}_{c\beta,m_v}}^{c\beta}\right\} =: C_K
\end{align*}
for all $g,h$. Hence, $\norm{K}{L_{p,v}\to L_{r,v}} \leq C_K$.\par
If $1< p<r=\infty$ and $g\in L_p(X,\mu)$ arbitrary, it follows with H\"older's inequality that
\begin{align*}
\norm{K(g)}{L_{\infty,v}} &\leq \esssup_{x\in X}\int_X|K(x,y)m_v(x,y)|\cdot|g(y)v(y)|\,\mathrm{d}\mu(y) \\
&\leq \norm{K}{\mathcal{A}_{p',m_v}}^{p'}\norm{g}{L_{p,v}}^p,
\end{align*}
which concludes the proof.
\end{proof}

\begin{remark}\label{remark:kernel_property}
\begin{thmlist}
\item The assumptions in \prettyref{lemma:kernel_property} can be weakened in the sense, that we only need specific $q>1$ for the assertion to hold, but this setting is sufficient for our work.
\item The proof is similar to the proof of Schur's test, also known as the generalized Young inequality. By letting $K\in\mathcal{A}_{1,m_v}$ and $p=r$ it follows that $1/\gamma=0$ and $a\alpha=c\beta=1$. This means we only use the two-way Young inequality and we are in the setting of Schur's test, see~\prettyref{lemma:schurs_test}.
\end{thmlist}
\end{remark}

\subsection{Coorbit spaces}

Before introducing coorbit spaces the concept of signals can first be generalized from elements of the Hilbert space $\cH$ to a suitable space of distributions.
First of all, for $1\leq\tau\leq 2$ consider the spaces \[ \cH_{\tau,v}:=\{f\in\cH,V_\frF f\in L_{\tau,v}(X,\mu)\} \] of test functions equipped with the natural norm \[ \norm{f}{\cH_{\tau,v}}:=\norm{V_\frF f}{L_{\tau,v}}. \]
First we note, that these spaces are non-empty, moreover the following Lemma holds.

\begin{lemma}\label{lemma:frame_subset}
If $R_\frF\in\cA_{\tau,m_v}$, then $\frF\subset\cH_{\tau,v}$.
\end{lemma}

\begin{proof}
For $x\in X$ arbitrary one has
\begin{align*}
\norm{\psi_x}{\cH_{\tau,v}}^\tau &= \int_X |V_\frF\psi_x(y)|^\tau v(y)^\tau\,\mathrm{d}\mu(y) \\
&\leq v(x)^\tau\int_X |R_\frF(y,x)|^\tau m_v(y,x)^\tau\,\mathrm{d}\mu(y) \\
&\leq v(x)^\tau\norm{R_\frF}{\cA_{\tau,m_v}}^{\tau},
\end{align*}
which proves the assertion.
\end{proof}

Since $\frF$ establishes a frame for $\cH$ this means $\cH_{\tau,v}\subset\cH$ is dense. Moreover, the spaces $\cH_{\tau,v}$ are Banach spaces, as the following Lemma states.

\begin{lemma}\label{lemma:h_banach}
If $R_\frF\in\cA_{\tau',m_v}$ then the space $\cH_{\tau,v}$ is a Banach space.
\end{lemma}

\begin{proof}
Let $\{f_n\}_{n\in\N}\subset\cH_{\tau,v}\subset\cH$ be a Cauchy sequence, which means $\{g_n\}_{n\in\N}:=\{V_\frF f_n\}_{n\in\N}$ is a Cauchy sequence in $L_{\tau,v}(X,\mu)$.
By the completeness of $L_{\tau,v}$ there exists a unique $g\in L_{\tau,v}$ with $g_n\to g$.
Furthermore, by the reproducing formula it holds $R_\frF(g_n)=g_n$ for all $n\in\N$, which implies $R_\frF(g)=g$.
Then, by H\"older's inequality, for every $x\in X$ it holds
\begin{align*}
|R_\frF(g)(x)| &\leq \int_X|R_\frF(x,y)g(y)|\,\mathrm{d}\mu(y) \\
&\leq \norm{R_\frF(x,\cdot)}{L_{\tau',\frac{1}{v}}}\cdot\norm{g}{L_{\tau,v}} \\
&\leq v(x)^{-1}\norm{R_\frF}{\cA_{\tau',m_v}}\cdot\norm{g}{L_{\tau,v}}.
\end{align*}
Thus, $g=R_\frF(g)\in L_\infty$ and since $L_\infty\cap L_{\tau,v}\subset L_2$ it follows $g\in L_2$. Since the application of $R_\frF$ is the orthogonal projection from $L_2$ onto the image of $V_\frF$ there exists $f\in\cH$ such that $g=V_\frF f$. Moreover, $V_\frF f\in L_{\tau,v}$ means $f\in\cH_{\tau,v}$ and $f_n\to f\in\cH_{\tau,v}$.
\end{proof}

Hence, this set of test functions leads to the Gelfand triple setting of dense embeddings \[ \cH_{\tau,v} \hookrightarrow \cH \cong \cH^\sim \hookrightarrow (\cH_{\tau,v})^\sim \] with $(\cH_{\tau,v})^\sim$ being the canonical anti-dual space (the space of all conjugate linear, continuous functionals) of $\cH_{\tau,v}$ and this space can be interpreted as a space of distributions.
An element $h\in(\cH_{\tau,v})^\sim$ is hereby identified with the functional $f\to\ip{h}{f}$.
With these embeddings it is possible to extend the notion of the voice transform in a canonical way to elements $f\in(\cH_{\tau,v})^\sim$ by $V_{\frF,\tau}f(x)=f(\psi_x)$.
By \prettyref{lemma:frame_subset} this is well defined.\par 
With assumptions on the reproducing kernel we can prove the following nesting property.

\begin{lemma}\label{lemma:nesting}
If $R_\frF\in\cA_{q,m_v}$ for every $q>1$ then $\cH_{\sigma,v}\subset\cH_{\tau,v}$ and $(\cH_{\tau,v})^\sim\subset(\cH_{\sigma,v})^\sim$ for all $\sigma<\tau$.
\end{lemma}

\begin{proof}
Assume $f\in\cH_{\sigma,v}$, which means $f\in\cH$ with $V_\frF f\in L_{\sigma,v}$. Since the reproducing identity holds it follows $V_\frF f=R_\frF(V_\frF f)\in R_\frF(L_{\sigma,v})$ and with \prettyref{lemma:kernel_property} we derive $V_\frF f\in L_{\tau,v}$, hence $f\in\cH_{\tau,v}$. The second assertion is immediate.
\end{proof}

For the coorbit spaces to be well defined we need the following two auxiliary Lemmas.

\begin{lemma}\label{lemma:equivalent_norm}
The expression $\norm{V_{\frF,\tau}f}{L_{\tau',\frac{1}{v}}(X,\mu)}$ is an equivalent norm on $(\cH_{\tau,v})^\sim$, where $\tau'$ denotes the H\"older-dual of $\tau$.
\end{lemma}

\begin{proof}
First we note that $V_{\frF}$ is acting as a unitary operator on $\cH$, and so does $V_{\frF,\tau}$.
Moreover, by definition we have $V_{\frF,\tau}(\cH_{\tau,v})=L_{2}\cap L_{\tau,v}$, which is dense in $L_{\tau,v}$.
Then, by definition of the norm one has
\begin{align*}
\norm{F}{(\cH_{\tau,v})^\sim} &= \sup_{\substack{h\in\cH_{\tau,v} \\ \norm{h}{\cH_{\tau,v}}\leq 1}}\abs{\ip{F}{h}} \\
&= \sup_{\substack{h\in\cH_{\tau,v} \\ \norm{V_{\frF,\tau}h}{L_{q,v}}\leq 1}}\abs{\ip{V_{\frF,\tau}F}{V_{\frF,\tau}h}} \\
&= \sup_{\substack{H\in V_{\frF,\tau}(\cH_{\tau,v}) \\ \norm{H}{L_{q,v}}\leq 1}}\abs{\ip{V_{\frF,\tau}F}{H}} \\
&= \sup_{\substack{H\in L_{\tau,v} \\ \norm{H}{L_{\tau,v}}\leq 1}}\abs{\ip{V_{\frF,\tau}F}{H}} \\
&= \norm{V_{\frF,\tau}F}{L_{\tau',\frac{1}{v}}},
\end{align*}
which concludes the proof.
\end{proof}

\begin{lemma}\label{lemma:h_properties}
\begin{thmlist}
\item For $f\in(\cH_{\tau,v})^\sim$ it holds $V_{\frF,\tau}f\in L_{\tau',\frac{1}{v}}$ and the mappings \mbox{$V_{\frF,\tau}:(\cH_{\tau,v})^\sim\to L_{\tau',\frac{1}{v}}$} are injective.
\item The reproducing formula extends to $(\cH_{\tau,v})^\sim$, i.e. $R_{\frF}(V_{\frF,\tau}f)=V_{\frF,\tau}f$ for all $f\in(\cH_{\tau,v})^\sim$.
\item Conversely, if $F\in L_{\tau',\frac{1}{v}}$ satisfies the reproducing property $R_{\frF}(F)=F$ then there exists $f\in(\cH_{\tau,v})^\sim$ such that $V_{\frF,\tau}f=F$.
\end{thmlist}
\end{lemma}

\begin{proof}
(i) The assertion follows immediately from \prettyref{lemma:equivalent_norm}.\\
(ii) Suppose that $f\in(\cH_{\tau,v})^\sim$.
Since $X$ is $\sigma$-compact there exists a sequence of nested compact subsets $(U_n)_{n\in\N}$ such that $X=\bigcup_{n\in\N}U_n$.
Denote by $\chi_{U_n}$ the characteristic function of $U_n$ and let $F_n:=\chi_{U_n}V_{\frF,\tau}f\in L_2$.
Obviously this series converges pointwise to $V_{\frF,\tau}f$.
For any $x\in X$ we then have
\begin{align*}
R_\frF(x,y)F_n(y) = 
\begin{cases}
R_\frF(x,y)V_{\frF,\tau}f(y), & y\in U_n,\\
0 , & \mbox{else},
\end{cases}
\end{align*}
which means that $|R_\frF(x,y)F_n(y)|\leq|R_\frF(x,y)V_{\frF,\tau}f(y)|$ for all $y\in X$.
Furthermore the expression $R_\frF(x,\cdot)V_{\frF,\tau}f$ is $L_1$-integrabel and by H\"older's inequality we obtain the estimation
\begin{align*}
\norm{R_\frF(x,\cdot)V_{\frF,\tau}f}{L_1} &\leq \norm{R_\frF(x,\cdot)}{L_{\tau,v}}\,\norm{V_{\frF,\tau}f}{L_{\tau',\frac{1}{v}}} \\
&\leq v(x)\norm{R_\frF}{\cA_{\tau,m_v}}\,\norm{f}{(\cH_{\tau,v})^\sim}
\end{align*}
for every $x\in X$. Since the reproducing property holds for every $F_n$ and because of Lebesgue's convergence theorem we obtain
\begin{align*}
V_{\frF,\tau}f(x) &= \lim_{n\to\infty}F_n(x) = \lim_{n\to\infty}\int_X R_\frF(x,y)F_n(y)\,\mathrm{d}\mu(y)\\
& = \int_X R_\frF(x,y)V_{\frF,\tau}f(y)\,\mathrm{d}\mu(y) = R_\frF(V_{\frF,\tau}f)(x).
\end{align*}
(iii) The adjoint mapping of $V_{\frF,\tau}:\cH_{\tau,v}\to L_{\tau,v}$ is given by
\begin{align*}
V_{\frF,\tau}^\ast:L_{\tau',\frac{1}{v}}\to(\cH_{\tau,v})^\sim, \quad V_{\frF,\tau}^\ast F=\int_X F(x)\psi_x\,\mathrm{d}\mu(x) \quad \mbox{for }F\in L_{\tau',\frac{1}{v}}.
\end{align*}
Thus for $f:=V_{\frF,\tau}^\ast F\in(\cH_{\tau,v})^\sim$ it holds
\begin{align*}
F(y) = R_\frF F(y) = \int_X\ip{\psi_x}{\psi_y}F(x)\,\mathrm{d}\mu(x) = V_{\frF,\tau}V_{\frF,\tau}^\ast F(y) = V_{\frF,\tau}f(y)
\end{align*}
for every $y\in X$.
\end{proof}

Now we are ready to define the coorbit spaces.

\begin{definition}\label{definition:coorbit_spaces}
The coorbit spaces of $L_{p,v}(X,\mu)$ with respect to the frame $\frF=\{\psi_x\}_{x\in X}$ and the integrability parameter $\tau$ are defined as
\begin{align*}
\mathrm{Co}_{\frF,\tau}(L_{p,v}) := \left\{f\in(\cH_{\tau,v})^\sim:V_{\frF,\tau}f\in L_{p,v}(X,\mu)\right\}
\end{align*}
endowed with the natural norms
\begin{align*}
\norm{f}{\mathrm{Co}_{\frF,\tau}(L_{p,v})} := \norm{V_{\frF,\tau}f}{L_{p,v}}.
\end{align*}
\end{definition}

The following proposition is essential when dealing with coorbit spaces.

\begin{proposition}\label{proposition:co_properties}
Suppose that $R_\frF(L_{p,v})\subset L_{\tau',\frac{1}{v}}$.
\begin{thmlist}
\item A function $F\in L_{p,v}$ is of the form $V_{\frF,\tau}f$ for some $f\in\mathrm{Co}_{\frF,\tau}(L_{p,v})$ if and only if $R_\frF F=F$.
\item The spaces $(\mathrm{Co}_{\frF,\tau}(L_{p,v}),\norm{\cdot}{\mathrm{Co}_{\frF,\tau}(L_{p,v})})$ are Banach spaces.
\item The map $V_{\frF,\tau}:\mathrm{Co}_{\frF,\tau}(L_{p,v})\to L_{p,v}$ induces an isometric isomorphism between $\mathrm{Co}_{\frF,\tau}(L_{p,v})$ and the reproducing kernel space $\{F\in L_{p,v}:R_\frF F=F\}\subset L_{p,v}$.
\end{thmlist}
\end{proposition}

\begin{proof}
(i) Assume $f\in\mathrm{Co}_{\frF,\tau}(L_{p,v})$, then by definition $f\in(\cH_{\tau,v})^\sim$ and by \prettyref{lemma:h_properties} ii) the reproducing identity holds.
Conversely, if $F\in L_{p,v}$ satisfies $R_\frF F=F$ we deduce by our assumption $F\in L_{\tau',\frac{1}{v}}$.
\prettyref{lemma:h_properties} iii) implies that there exists $f\in(\cH_{\tau,v})^\sim$ such that $V_\frF f=F$, which shows the assertion.\\
(ii) Suppose that $\{f_n\}_{n\in\N}$ is a Cauchy sequence in $\mathrm{Co}_{\frF,\tau}(L_{p,v})$ implying that $F_n := V_{\frF,\tau} f_n$ is a Cauchy sequence in $L_{p,v}$.
By the completeness of $L_{p,v}$ this sequence convergences to an element $F\in L_{p,v}$.
By i) it holds $R_\frF F_n=F_n$ for all $n\in\N$ and hence $R_\frF F=F$.
Again by i) there exists an $f\in\mathrm{Co}_{\frF,q}(L_{p,v})$ with $V_{\frF,\tau}f=F$ and the completeness is shown.\\
(iii) The assertion follows with (i) and the injectivity of $V_{\frF,\tau}$.
\end{proof}

\begin{remark}\label{remark:assumption_fulfilled}
The assumption in \prettyref{proposition:co_properties} may appear strange, but is readily fulfilled for the following setting. If we assume $R_\frF\in\cA_{q,m_v}$ for all $q>1$ then it follows from \prettyref{lemma:kernel_property} that \mbox{$R_\frF(L_{p,v}) \subset L_{\tau',v} \subset L_{\tau',\frac{1}{v}}$} for all \mbox{$1<p<\tau'<\infty$}.
\end{remark}

\subsection{Dependency on $\tau$, $p$, $v$ and $\frF$}

We will now discuss the dependency of the coorbit spaces on the parameters involved.
To this end we always assume $R_\frF\in\cA_{q,m_v}$ for all $q>1$ as suggested in \prettyref{remark:assumption_fulfilled}.\par 
We can obtain some nesting properties for the parameters $\tau$ and $p$ as well as the weight $v$.

\begin{lemma}\label{lemma:embeddings}
\begin{thmlist}
\item For all $\sigma<\tau$ we have $\mathrm{Co}_{\frF,\tau}(L_{p,v})\subset\mathrm{Co}_{\frF,\sigma}(L_{p,v})$.
\item For all $p<r$ we have $\mathrm{Co}_{\frF,\tau}(L_{p,v})\subset\mathrm{Co}_{\frF,\tau}(L_{r,v})$.
\item For two weights fulfilling $v\leq w$ we have $\mathrm{Co}_{\frF,\tau}(L_{p,w})\subset\mathrm{Co}_{\frF,\tau}(L_{p,v})$.
\end{thmlist}
\end{lemma}

\begin{proof}
(i) This follows immediately from \prettyref{lemma:nesting}.\\
(ii) Assume $f\in\mathrm{Co}_{\frF,\tau}(L_{p,v})$, meaning $f\in(\cH_{\tau,v})^\sim$ with $V_{\frF,\tau}f\in L_{p,v}$.
By \prettyref{lemma:h_properties} (ii) the reproducing identity extends to $(\cH_{\tau,v})^\sim$, thus $V_{\frF,\tau}f=R_\frF(V_{\frF,\tau}f)\in R_\frF(L_{p,v})$.
With \prettyref{lemma:kernel_property} we derive $V_{\frF,\tau}f\in L_{r,v}$, which shows the assumption.\\
(iii) Since $L_{p,w}\subset L_{p,v}$ the assertion holds.
\end{proof}

\begin{remark}\label{remark:embeddings}
Under the additional assumption $R_\frF\in\cA_{1,m_v}$, the spaces $\mathrm{Co}_{\frF,1}(L_{p,v})$, which are analyzed in~Ref.~\cite{FR:InhCoorbit}, are well-defined by Schur's test, see~\prettyref{lemma:schurs_test}. Hence, by \prettyref{lemma:embeddings} (i) we have the embeddings $\mathrm{Co}_{\frF,\tau}(L_{p,v})\subset\mathrm{Co}_{\frF,1}(L_{p,v})$ for all $1<\tau\leq 2$. This is not applicable for the inhomogeneous shearlet coorbit spaces we are looking at in this paper but may be of interest for other spaces.
\end{remark}

To identify conditions under which the Coorbit spaces are independent of the frame, we introduce a second Parseval frame for $\cH$ we denote by $\frG=\{\tilde\psi_x\}_{x\in X}$ and introduce the Gramian kernel as
\begin{align*}
G(\frF,\frG)(x,y):=\ip{\tilde\psi_y}{\psi_x}.
\end{align*}
Then, the following holds true.

\begin{proposition}\label{proposition:equality_frf}
Assume that $\frF=\{\psi_x\}_{x\in X}$ and $\frG=\{\tilde\psi_x\}_{x\in X}$ are two Parseval frames for $\cH$ fulfilling all necessary conditions on the reproducing kernels and the corresponding Gramian kernel fulfills $G(\frF,\frG)\in\cA_{1,m_v}$.
Then it holds $\mathrm{Co}_{\frF,\tau}(L_{p,v})=\mathrm{Co}_{\frG,\tau}(L_{p,v})$.
\end{proposition}

\begin{proof}
By expanding $V_\frF$ with respect to $\frG$ we obtain
\begin{align*}
V_\frF f(x) = \ip{f}{\psi_x} = \int_X\ip{f}{\tilde\psi_y}\ip{\tilde\psi_y}{\psi_x}\,\mathrm{d}\mu(y) = G(\frF,\frG)(V_\frG f)(x)
\end{align*}
and the same holds for the extended voice transform.
By our assumption we derive with Schur's test that $G(\frF,\frG)(L_{p,v})\subset L_{p,v}$ and it holds
\begin{align*}
\norm{f}{\mathrm{Co}_{\frF,\tau}(L_{p,v})} \leq \norm{G(\frF,\frG)}{\cA_{1,m_v}}\,\norm{f}{\mathrm{Co}_{\frG,\tau}(L_{p,v})}.
\end{align*}
The converse is shown analogously and the assertion follows.
\end{proof}

\section{Shearlet coorbit spaces}\label{sec:shearletcoorbit}

In this section we introduce an inhomogeneous version of the shearlet transform and define smoothness spaces associated to this transform.
In order to accomplish this we use the generalized coorbit theory outlined in \prettyref{sec:coorbit_theory}.
Since our approach is based on the homogeneous shearlet transform and the resulting coorbit spaces (as treated in~Ref.~\cite{Dahlke:TracesEmbeddings,ShearletCoorbits,Dahlke:ShearletBanachFrames,Dahlke:ShearletArbitrary}), we start by giving a short overview of the respective theory.
By modifying the homogeneous shearlet transform, we then develop a new transform, given through the action of an (inhomogeneous) frame.
For this new transform we then show that all the necessary conditions on the reproducing kernel hold, so that we can introduce the associated coorbit spaces with respect to the (weighted) Lebesgue spaces.

\subsection{Homogeneous shearlet transform}\label{sec:homogeneousshearlettransform}
To define the shearlet transform, one starts with an \emph{admissible} function $\psi \in L_2(\R^d)$, i.e. a function satisfying the condition
\begin{equation}\label{eq:admissibility}
c_\psi := \int_{\R^d} \frac{\abs{\hat{\psi}(\omega)}^2}{\abs{\omega_1}^d}\,\mathrm{d}\omega < \infty.
\end{equation}
This condition is necessary for the transform to be square-integrable.
The admissible function is then translated, dilated and sheared in order to change its localization, scale and orientation.
For a parameter $a \in \R^*$ let \[ A_a = \begin{pmatrix}
a & 0_{d-1}^T\\0_{d-1} & \sgn(a)\abs{a}^{\frac1d} I_{d-1}
\end{pmatrix} \] denote a generalized parabolic scaling matrix and for a parameter $s \in \R^{d-1}$ let \[ S_s = \begin{pmatrix}
1 & s^T\\0_{d-1} & I_{d-1}
\end{pmatrix} \] denote the so-called shear matrix.
It is easy to see that $\abs{\det S_s} = 1$ and $\abs{\det A_a} = \abs{a}^{2-\frac1d}$.
Using these matrices one can then define the translated, dilated and sheared version of $\psi$ through \[ \psi_{(a,s,t)}(x) = \abs{\det A_a}^{-\frac12} \psi(A_a^{-1} S_s^{-1}(x-t)).\]
In the homogeneous setting, the shearlet transform is then defined through the action of a unitary, irreducible and integrable representation of the full parameter group, the so-called shearlet group $\mathds{S} = \R^* \times \R^{d-1} \times \R^d$ with the group law \[ (a,s,t) \circ (a',s',t') = (aa',s + \abs{a}^{1-\frac1d} s',t + S_s A_a t'). \]
Given the mapping $\pi : \mathds{S} \rightarrow \cU(L_2(\R^d))$ with $\pi(a,s,t)\psi = \psi_{(a,s,t)}$, which can be shown to be a unitary group representation, the shearlet transform is defined as \[ \mathcal{SH}: L_2(\R^d) \rightarrow L_2(\mathds{S}),\quad f \mapsto \mathcal{SH}f \] with \[ \mathcal{SH}f: \mathds{S} \rightarrow \C,\quad (a,s,t) \mapsto \ip{f}{\pi(a,s,t)\psi}_{L_2(\R^d)}. \]
Based on this notion of the shearlet transform Dahlke et al. introduced homogeneous shearlet coorbit spaces with respect to the Lebesgue spaces by using the coorbit space theory developed by Feichtinger and Gr\"ochenig in Ref.~\cite{FG:Coorbit1,FG:Coorbit2,FG:Coorbit3}.

\subsection{Inhomogeneous shearlet frame}\label{sec:inhFrame}

Similar to the wavelet approach in Ref.~\cite{RauhutUllrich} we now introduce an inhomogeneous shearlet transform by restricting the dilation parameter to a closed subset of the full parameter group, thereby only covering the higher-frequency content of a signal.
To analyze the polynomial and lower-frequency part a second function is introduced to construct an inhomogeneous frame of functions in $L_2(\R^d)$ as the set of building blocks for our new transform.
Therefore we choose the set \[ X := \Bigl( \{\infty\}\times \R^{d-1} \times \R^d \Bigr) \cup \Bigl( [-1,1]^* \times \R^{d-1} \times \R^d \Bigr) \] as the new parameter space with ``$\infty$'' representing an isolated point in $\R$ and $[-1,1]^* := [-1,1]\setminus \{0\}$.
The right-hand side of the union is the aforementioned subspace of the shearlet group $\mathds{S}$, which is closed under the group action.
Obviously, this definition leads to a locally compact Hausdorff space.
In the following definition we introduce a measure on the parameter space so that $X$, together with its Borel $\sigma$-algebra, becomes a measure space.

\begin{definition}\label{definition:measure}
On the space $X$ a measure $\mu$ is defined by
\begin{equation}\label{eq:measure}
\integral{X}{F(x)}{\mu(x)} := \measureX{F(\infty,s,t)}{F(a,s,t)}
\end{equation}
with $F$ being a complex-valued function on $X$ which is measurable with respect to the Borel $\sigma$-algebra.
\end{definition}

The first summand in the definition above is composed of the point measure on $\R$ and the Lebesgue measure on $\R^{d-1} \times \R^d$, while the second summand is the restriction of the (left) Haar measure on the shearlet group to the subset $[-1,1]^*\times \R^{d-1}\times \R^d$.
Therefore it is obvious that $\mu$ given by~\prettyref{eq:measure} is a positive Radon measure.
Choosing the measure space $(X,\mathfrak{B}(X),\mu)$ as the underlying index space, we can introduce a continuous shearlet frame.

\begin{definition}
Let $a \in \R^*$, $s \in \R^{d-1}$ and $t \in \R^d$.
Then
\begin{thmlist}
\item $L_t: L_2(\R^d)\rightarrow L_2(\R^d)$ with $L_t \psi := \psi (\cdot - t)$ is called the \emph{(left) translation operator},
\item $D_{S_s}: L_2(\R^d)\rightarrow L_2(\R^d)$ with $D_{S_s} \psi := \psi (S_s^{-1} \cdot)$ is called the \emph{shearing operator}, and
\item $D_{A_a}: L_2(\R^d)\rightarrow L_2(\R^d)$ with $D_{A_a} \psi := \abs{\det A_a}^{-\frac{1}{2}} \psi(A_a^{-1} \cdot)$ is called the \emph{(anisotropic) dilation operator}.
\end{thmlist}
\end{definition}

Using the above defined operators, we can define an \emph{inhomogeneous shearlet frame}.

\begin{definition}\label{definition:shearletframe}
Let $\Phi, \Psi \in L_2(\R^d)$ with $\Psi$ being an admissible shearlet.
Then we define $\frF := \{ \psi_x \}_{x \in X}$ with
\begin{align}
&\psi_{(\infty,s,t)} := L_t D_{S_s} \Phi = \Phi (S_s^{-1} (\cdot - t))\mbox{ and}\label{eq:generator}\\
&\psi_{(a,s,t)} := L_t D_{S_s} D_{A_a} \Psi = \abs{\det A_a}^{-\frac{1}{2}} \Psi(A_a^{-1} S_s^{-1} (\cdot - t)).\label{eq:shearlet}
\end{align}
\end{definition}

The main theorem of this section is that $\frF$, given by~\prettyref{eq:generator} and~\prettyref{eq:shearlet}, constitutes a continuous Parseval frame under the conditions given in~\prettyref{theorem:tightframe} below so that the transform based on $\frF$ is well defined.
To this end we need two technical results that can also be found in~Ref.~\cite{Dahlke:ShearletArbitrary}.

\begin{lemma}\label{lemma:convolution}
For all $(\alpha,s,t) \in X$ with $\alpha = a$ or $\alpha = \infty$ and $f,\psi \in L_2(\R^d)$ the identity
\begin{equation*}
\ip{f}{\psi_{(\alpha,s,t)}}_{L_2(\R^d)} = (f * \psi^{*}_{(\alpha,s,0)})(t)
\end{equation*}
holds true with $\psi^{*} := \overline{\psi (-\cdot)}$. 
\end{lemma}

\begin{lemma}\label{lemma:fourier}
Let $\phi \in L_2(\R^d)$, $a \in \R^*$, $s \in \R^{d-1}$ and $\xi \in \R^d$.
Then the following equations hold:
\begin{thmlist}
\item\label{list:fouriergenerator} $\cF(D_{S_s} \phi) (\xi) = \hat{\phi}(S^T_s \xi)$;
\item\label{list:fouriershearlet} $\cF(D_{S_s} D_{A_a} \phi) (\xi) = \abs{\det A_a}^{\frac{1}{2}} \hat{\phi}(A_a S^T_s \xi)$.
\end{thmlist}
\end{lemma}

We now state the main theorem of this section, which identifies conditions on $\Phi$ and $\Psi$ for $\frF$ being a continuous Parseval frame.

\begin{theorem}\label{theorem:tightframe}
Let $\Psi \in L_1(\R^d)\cap L_2(\R^d)$ be an admissible shearlet and let $\Phi \in L_1(\R^d)\cap L_2(\R^d)$ be such that
\begin{equation}\label{eq:assumptions}
\integral{\R^{d-1}}{\frac{\abs{\hat{\Phi}(y,\sigma)}^2}{\abs{y}^{d-1}}}{\sigma} + \integral{\R^{d-1}}{\integral[\abs{y}]{-\abs{y}}{\frac{\abs{\hat{\Psi}(\xi_1,\tilde{\xi})}^2}{\abs{\xi_1}^d}}{\xi_1}}{\tilde{\xi}} = 1\quad\text{for almost every }y\in \R.
\end{equation}
Then the inhomogeneous shearlet frame $\frF$ is a continuous Parseval frame of $L_2(\R^d)$, i.e.,
\begin{equation*}
\integral{X}{\abs{\ip{f}{\psi_x}}^2}{\mu(x)} = \norm{f}{L_2(\R^d)}^2,\quad f\in L_2(\R^d).
\end{equation*}
\end{theorem}

\begin{proof}
Applying~\prettyref{eq:measure}, Fubini's and Plancherel's theorem we obtain
\begin{align*}
\integral{X}{\abs{\ip{f}{\psi_x}}^2}{\mu(x)} &= \integral{\R^d}{\integral{\R^{d-1}}{\abs{\ip{f}{\psi_{(\infty,s,t)}}}^2}{s}}{t}\\
&\qquad\qquad + \integral{\R^d}{\integral{\R^{d-1}}{\int\limits^1_{-1} \abs{\ip{f}{\psi_{(a,s,t)}}}^2\,\frac{\mathrm{d}a}{\abs{a}^{d+1}}}{s}}{t}\\
&= \integral{\R^{d-1}}{\integral{\R^d}{\abs{\ip{f}{\psi_{(\infty,s,t)}}}^2}{t}}{s}\\
&\qquad\qquad + \integral{\R^{d-1}}{\int\limits^{1}_{-1} \integral{\R^d}{\abs{\ip{f}{\psi_{(a,s,t)}}}^2}{t}\,\frac{\mathrm{d}a}{\abs{a}^{d+1}}}{s}\\
&= \integral{\R^{d-1}}{\norm{\ip{f}{\psi_{(\infty,s,\cdot)}}}{L_2(\R^d)}^2}{s}\\
&\qquad\qquad + \integral{\R^{d-1}}{\int\limits^1_{-1} \norm{\ip{f}{\psi_{(a,s,\cdot)}}}{L_2(\R^d)}^2\,\frac{\mathrm{d}a}{\abs{a}^{d+1}}}{s}\\
&= \integral{\R^{d-1}}{\norm{\cF(\ip{f}{\psi_{(\infty,s,\cdot)}})}{L_2(\R^d)}^2}{s}\\
&\qquad\qquad + \integral{\R^{d-1}}{\int\limits^1_{-1} \norm{\cF(\ip{f}{\psi_{(a,s,\cdot)}})}{L_2(\R^d)}^2\,\frac{\mathrm{d}a}{\abs{a}^{d+1}}}{s}\\
&= \integral{\R^{d-1}}{\!\integral{\R^d}{\abs{\cF(\ip{f}{\psi_{(\infty,s,\cdot)}})(t)}^2}{t}}{s}\\
&\qquad\qquad + \integral{\R^{d-1}}{\int\limits^{1}_{-1} \integral{\R^d}{\abs{\cF(\ip{f}{\psi_{(a,s,\cdot)}})(t)}^2}{t}\,\frac{\mathrm{d}a}{\abs{a}^{d+1}}}{s}.
\end{align*}
Using~\prettyref{lemma:convolution}, Fubini's theorem, the fact that $\cF(f * g) = \hat{f} \hat{g}$ and $\abs{\cF(f^*)} = \abs{\cF(f)}$ leads to
\begin{align*}
\lefteqn{\integral{X}{\abs{\ip{f}{\psi_x}}^2}{\mu(x)} = \integral{\R^{d-1}}{\integral{\R^d}{\abs{\cF(f * \psi^{*}_{(\infty,s,0)})(t)}^2}{t}}{s}}\\
&\qquad\qquad\qquad\qquad\qquad\qquad + \integral{\R^{d-1}}{\int\limits^1_{-1} \integral{\R^d}{\abs{\cF(f * \psi^{*}_{(a,s,0)})(t)}^2}{t}\,\frac{\mathrm{d}a}{\abs{a}^{d+1}}}{s}\\
&= \integral{\R^{d-1}}{\integral{\R^d}{\abs{\hat{f}(t)}^2 \abs{\cF(\psi^{*}_{(\infty,s,0)})(t)}^2}{t}}{s}\\
&\qquad\qquad\qquad\qquad + \integral{\R^{d-1}}{\int\limits^1_{-1} \integral{\R^d}{\abs{\hat{f}(t)}^2 \abs{\cF(\psi^{*}_{(a,s,0)})(t)}^2}{t}\,\frac{\mathrm{d}a}{\abs{a}^{d+1}}}{s}\\
&= \integral{\R^d}{\abs{\hat{f}(t)}^2 \biggl( \integral{\R^{d-1}}{\abs{\cF(\psi_{(\infty,s,0)}) (t)}^2}{s} + \integral{\R^{d-1}}{\int\limits^1_{-1} \abs{\cF(\psi_{(a,s,0)})(t)}^2\,\frac{\mathrm{d}a}{\abs{a}^{d+1}}}{s} \biggr)}{t}.
\end{align*}
Thus, if we can prove that
\begin{equation}\label{eq:condtightframe}
\integral{\R^{d-1}}{\abs{\cF(\psi_{(\infty,s,0)}) (t)}^2}{s} + \integral{\R^{d-1}}{\int\limits^1_{-1} \abs{\cF(\psi_{(a,s,0)})(t)}^2\,\frac{\mathrm{d}a}{\abs{a}^{d+1}}}{s} \stackrel{!}{=} 1
\end{equation}
for almost every $t\in\R^d$,the assertion follows, since then
\begin{equation*}
\integral{X}{\abs{\ip{f}{\psi_x}}^2}{\mu(x)} = \integral{\R^d}{\abs{\hat{f}(t)}^2}{t} = \norm{\hat{f}}{L_2(\R^d)}^2 = \norm{f}{L_2(\R^d)}^2.
\end{equation*}
Hence, it remains to show~\prettyref{eq:condtightframe}.
Assuming that $t_1 \neq 0$ we use~\prettyref{lemma:fourier} to obtain
\begin{align*}
\lefteqn{\integral{\R^{d-1}}{\abs{\cF(\psi_{(\infty,s,0)}) (t)}^2}{s} + \integral{\R^{d-1}}{\int\limits^1_{-1} \abs{\cF(\psi_{(a,s,0)})(t)}^2\,\frac{\mathrm{d}a}{\abs{a}^{d+1}}}{s}}\\
&= \integral{\R^{d-1}}{\abs{\cF(D_{S_s} \Phi) (t)}^2}{s} + \integral{\R^{d-1}}{\int\limits^1_{-1} \abs{\cF(D_{S_s} D_{A_a} \Psi) (t)}^2\,\frac{\mathrm{d}a}{\abs{a}^{d+1}}}{s}\\
&= \integral{\R^{d-1}}{\abs{\hat{\Phi}(S^T_s t)}^2}{s} + \integral{\R^{d-1}}{\int\limits^1_{-1} \abs{\det A_a} \abs{\hat{\Psi}(A_a S_s^T t)}^2\,\frac{\mathrm{d}a}{\abs{a}^{d+1}}}{s}\\
&= \integral{\R^{d-1}}{\abs{\hat{\Phi}(t_1,\tilde{t} + t_1 s)}^2}{s} + \integral{\R^{d-1}}{\int\limits^1_{-1} \abs{\det A_a} \abs{\hat{\Psi}(a t_1,\sgn(a) \abs{a}^{\frac{1}{d}} (\tilde{t} + t_1 s))}^2\,\frac{\mathrm{d}a}{\abs{a}^{d+1}}}{s},
\end{align*}
with $t = (t_1,\tilde{t})^T$, $\tilde{t} \in \R^{d-1}$.
Substituting $ \sigma := \tilde{t} + t_1 s \mbox{ and } \xi = (\xi_1,\tilde{\xi}) := (a t_1,\sgn(a) \abs{a}^{\frac{1}{d}} (\tilde{t} + t_1 s)),$ we end up with
\begin{align*}
\lefteqn{\integral{\R^{d-1}}{\abs{\cF(\psi_{(\infty,s,0)}) (t)}^2}{s} + \integral{\R^{d-1}}{\int\limits^1_{-1} \abs{\cF(\psi_{(a,s,0)})(t)}^2\,\frac{\mathrm{d}a}{\abs{a}^{d+1}}}{s}}\\
&\hspace{2cm}= \integral{\R^{d-1}}{\abs{t_1}^{-(d-1)} \abs{\hat{\Phi}(t_1,\sigma)}^2}{\sigma} + \integral{\R^{d-1}}{\integral[\abs{t_1}]{-\abs{t_1}}{\abs{\xi_1}^{-d} \abs{\hat{\Psi}(\xi_1,\tilde{\xi})}^2}{\xi_1}}{\tilde{\xi}}\\
&\hspace{2cm}= \integral{\R^{d-1}}{\frac{\abs{\hat{\Phi}(t_1,\sigma)}^2}{\abs{t_1}^{d-1}}}{\sigma} + \integral{\R^{d-1}}{\integral[\abs{t_1}]{-\abs{t_1}}{\frac{\abs{\hat{\Psi}(\xi_1,\tilde{\xi})}^2}{\abs{\xi_1}^d}}{\xi_1}}{\tilde{\xi}},
\end{align*}
and~\prettyref{eq:condtightframe} follows from assumption~\prettyref{eq:assumptions}.
\end{proof}

\begin{remark}
The proof of~\prettyref{theorem:tightframe} can also be stated in a similar manner for the case of a tight frame with arbitrary frame constant $A < \infty$.
The only difference is that $\Phi$ and $\Psi$ have to satisfy
\begin{equation*}
\integral{\R^{d-1}}{\frac{\abs{\hat{\Phi}(y,\sigma)}^2}{\abs{y}^{d-1}}}{\sigma} + \integral{\R^{d-1}}{\!\integral[\abs{y}]{-\abs{y}}{\frac{\abs{\hat{\Psi}(\xi_1,\tilde{\xi})}^2}{\abs{\xi_1}^d}}{\xi_1}}{\tilde{\xi}} = A\quad\text{for almost every }y\in \R
\end{equation*}
instead of~\prettyref{eq:assumptions}.
\end{remark}

\begin{remark}\label{remark:choicePhi}
For a given shearlet $\Psi$ it is still necessary to show that one can satisfy condition~\prettyref{eq:assumptions} for a function $\Phi \in L_1(\R^d) \cap L_2(\R^d)$.
To this end we restrict ourselves to odd dimensions and we define $\hat\Phi:\R^d\to\C$ by
\begin{equation*}
\hat\Phi(\xi) := \xi_1^{\frac{d-1}{2}}\biggl(\integral{\R\setminus[-\abs{\xi_1},\abs{\xi_1}]}{\frac{\abs{\hat\Psi(\omega_1,\tilde\xi)}^2}{\abs{\omega_1}^d}}{\omega_1}\biggr)^{1/2}.
\end{equation*}
It is straightforward to see that $\Phi$ fulfills \prettyref{eq:assumptions}. Moreover, $\Phi\in L_2(\R^d)$ is immediate and $\Phi\in L_1(\R^d)$ can be shown if $\hat\Phi\in\mathscr{C}^\infty_0(\R^d)$, see \prettyref{example:psiphi}.
\end{remark}

Because of~\prettyref{theorem:tightframe}, we can now state the definition of the shearlet transform based on $\frF$.

\begin{definition}
Let $\Phi,\Psi \in L_2(\R^d)$ satisfy the assumptions of~\prettyref{theorem:tightframe} and let $\frF = \{ \psi_x \}_{x\in X}$ be given by~\prettyref{definition:shearletframe}.
Then the {\em shearlet transform} based on $\frF$ is defined as \[ \mathcal{SH}_\frF: L_2(\R^d) \rightarrow L_2(X,\mu), f \mapsto \mathcal{SH}_\frF f \] with \[ \mathcal{SH}_\frF f: X \rightarrow \C, x \mapsto \ip{f}{\psi_x}. \]
\end{definition}

\subsection{Conditions on the reproducing kernel}\label{sec:conditions}

The main goal of this section is to lay the foundations for the definition of the coorbit spaces $\mathrm{Co}_{\frF,\tau}(L_{p,v}(X,\mu))$, $1\leq p<\infty$, $p<\tau'<\infty$, with $v$ being a weight function on $X$, associated to the inhomogeneous shearlet transform introduced in the previous section.
To prove that these spaces are well-defined Banach spaces, we need to show that the conditions on $\frF$, as stated in~\prettyref{sec:coorbit_theory}, are satisfied.
By \prettyref{remark:assumption_fulfilled} it suffices to show that $R_\frF\in\cA_{q,m_v}$ for all $q>1$.
To this end we need the following auxiliary results.

\begin{lemma}
Let $a,a' \in [-1,1]^*,\ s,s' \in \R^{d-1},\ t,t' \in \R^d$ and $\varphi_{(a,s,t)} := \abs{\det A_a}^{-\frac12} \Phi(A_a^{-1} S_s^{-1} (\cdot - t))$.
It follows that
\begin{equation}\label{eq:IP1}
\abs{\ip{\psi_{(\infty,s,t)}}{\psi_{(\infty,s',t')}}} = \abs{(\mathcal{SH} \Phi)(\infty,s-s',S_{s'}^{-1}(t-t'))},
\end{equation}
\begin{equation}\label{eq:IP2}
\abs{\ip{\psi_{(\infty,s,t)}}{\psi_{(a',s',t')}}} = \abs{\ip{\Psi}{\varphi_{(a'^{-1},\abs{a'}^{\frac1d-1}(s-s'),A_{a'}^{-1} S_{s'}^{-1}(t-t')}}},
\end{equation}
\begin{equation}\label{eq:IP3}
\abs{\ip{\psi_{(a,s,t)}}{\psi_{(\infty,s',t')}}} = \abs{(\mathcal{SH} \Phi)(a,s-s',S_{s'}^{-1}(t-t'))},
\end{equation}
\begin{equation}\label{eq:IP4}
\abs{\ip{\psi_{(a,s,t)}}{\psi_{(a',s',t')}}} = \abs{(\mathcal{SH} \Psi)(aa'^{-1},\abs{a'}^{\frac1d-1}(s-s'),A_{a'}^{-1} S_{s'}^{-1} (t-t'))}.
\end{equation}
\end{lemma}

\begin{proof}
We only state the proof for~\prettyref{eq:IP4} in detail, \prettyref{eq:IP1}--\prettyref{eq:IP3} can be proven analogously.
By the definition of $\psi_{(a,s,t)}$ we obtain
\begin{align*}
\lefteqn{\ip{\psi_{(a,s,t)}}{\psi_{(a',s',t')}} = \integral{\R^d}{\psi_{(a,s,t)}(x) \overline{\psi_{(a',s',t')}(x)}}{x}}\\
&\qquad= \integral{\R^d}{\abs{\det A_a}^{-\frac12}\Psi(A_a^{-1} S_s^{-1}(x - t)) \overline{\abs{\det A_{a'}}^{-\frac{1}{2}} \Psi(A_{a'}^{-1} S_{s'}^{-1}(x - t'))}}{x},
\end{align*}
which, by means of the substitution $y = A_{a'}^{-1} S_{s'}^{-1}(x-t')$, leads to
\begin{align*}
\lefteqn{\ip{\psi_{(a,s,t)}}{\psi_{(a',s',t')}} = \integral{\R^d}{\abs{\det A_{aa'^{-1}}}^{-\frac12} \Psi(A_a^{-1} S_s^{-1} (S_{s'} A_{a'} y + t' - t)) \overline{\Psi(y)}}{y}}\\
&\qquad= \integral{\R^d}{\abs{\det A_{aa'^{-1}}}^{-\frac12} \Psi(A_a^{-1} S_s^{-1} S_{s'} A_{a'}(y - (A_{a'}^{-1} S_{s'}^{-1} (t-t')))) \overline{\Psi(y)}}{y}\\
&\qquad= \integral{\R^d}{\abs{\det A_{aa'^{-1}}}^{-\frac12} \Psi(A_{aa'^{-1}}^{-1} S_{\abs{a'}^{\frac1d-1} (s-s')}^{-1}(y - (A_{a'}^{-1} S_{s'}^{-1} (t-t')))) \overline{\Psi(y)}}{y}\\
&\qquad= \ip{\psi_{(aa'^{-1},\abs{a'}^{\frac1d-1}(s-s'),A_{a'}^{-1} S_{s'}^{-1} (t-t'))}}{\Psi}.
\end{align*}
This yields
\begin{align*}
\abs{\ip{\psi_{(a,s,t)}}{\psi_{(a',s',t')}}} &= \abs{\ip{\psi_{(aa'^{-1},\abs{a'}^{\frac1d-1}(s-s'),A_{a'}^{-1} S_{s'}^{-1} (t-t'))}}{\Psi}}\\
&= \abs{\ip{\Psi}{\psi_{(aa'^{-1},\abs{a'}^{\frac1d-1}(s-s'),A_{a'}^{-1} S_{s'}^{-1} (t-t'))}}}\\
&= \abs{(\mathcal{SH} \Psi)(aa'^{-1},\abs{a'}^{\frac1d-1}(s-s'),A_{a'}^{-1} S_{s'}^{-1} (t-t'))}.
\end{align*}
\end{proof}

Using the auxiliary result above, we can prove the following lemma concerning the $\cA_{q,m_v}$-Norm of $R_\frF$.

\begin{lemma}\label{lemma:kernel}
Let $R_\frF$ be the kernel function associated to the inhomogeneous shearlet frame as defined by~\prettyref{eq:kernel}.
Then for every $q$ the following identity holds:
\begin{align} \label{eq:kernel}
\lefteqn{\esssup_{(\alpha,\sigma,\tau) \in X} \integral{X}{\abs{R_\frF ((\alpha,\sigma,\tau),(a,s,t))}^q m_v((\alpha,\sigma,\tau),(a,s,t))^q}{\mu (a,s,t)}} \notag\\
&= \max \Biggl\{ \esssup_{(\sigma,\tau)\in \R^{d-1}\times \R^d} \int\limits_{\R^d}\int\limits_{\R^{d-1}}\biggl( \max\left\{\frac{v(\infty,\sigma,\tau)}{v(\infty,\tilde{\sigma_1},\tilde{\tau_1})},\frac{v(\infty,\tilde{\sigma_1},\tilde{\tau_1})}{v(\infty,\sigma,\tau)}\right\}^q \abs{\ip{\Phi}{\psi_{(\infty,s',t')}}}^q \notag\\
&\hspace{1cm}+ \int\limits_{-1}^{1} \max\left\{\frac{v(\infty,\sigma,\tau)}{v(a,\tilde{\sigma_2},\tilde{\tau_2})},\frac{v(a,\tilde{\sigma_2},\tilde{\tau_2})}{v(\infty,\sigma,\tau)}\right\}^q\abs{\ip{\Phi}{\psi_{(a,s',t')}}}^q\,\frac{\mathrm{d}a}{\abs{a}^{d+1}} \biggr)\,\mathrm{d}s'\,\mathrm{d}t',\notag\\
& \esssup_{(\alpha,\sigma,\tau) \in [-1,1]^*\times \R^{d-1}\times\R^d} \int\limits_{\R^d}\int\limits_{\R^{d-1}}\biggl(\max\left\{\frac{v(\alpha,\sigma,\tau)}{v(\infty,\tilde{\sigma_1},\tilde{\tau_1})},\frac{v(\infty,\tilde{\sigma_1},\tilde{\tau_1})}{v(\alpha,\sigma,\tau)}\right\}^q \abs{\ip{\Phi}{\psi_{(\alpha,s',t')}}}^q \\
&\hspace{1cm}+ \int\limits_{-\abs{\alpha}^{-1}}^{\abs{\alpha}^{-1}} \max\left\{\frac{v(\alpha,\sigma,\tau)}{v(\tilde{\alpha},\tilde{\sigma_3},\tilde{\tau_3})},\frac{v(\tilde{\alpha},\tilde{\sigma_3},\tilde{\tau_3})}{v(\alpha,\sigma,\tau)}\right\}^q\abs{\ip{\Psi}{\psi_{(a',s',t')}}}^q\,\frac{\mathrm{d}a'}{\abs{a'}^{d+1}} \biggr)\,\mathrm{d}s'\,\mathrm{d}t' \Biggr\} \notag
\end{align}
with $\tilde{\sigma_1} = \sigma-s', \tilde{\tau_1} = \tau - S_{\tilde{\sigma_1}}t', \tilde{\sigma_2} = \sigma+s', \tilde{\tau_2} = \tau + S_\sigma t', \tilde{\alpha} = \alpha a', \tilde{\sigma_3} = \sigma + \abs{\alpha}^{1-\frac1d}s', \tilde{\tau_3} = \tau + S_\sigma A_\alpha t'$.
\end{lemma}

\begin{proof}
Let $(\alpha,\sigma,\tau) \in X$ with $\alpha \in \{\infty\}\cup [-1,1]^*$.
Using~\prettyref{eq:IP1} and~\prettyref{eq:IP3} we obtain
\begin{align*}
\integral{\R^d}{\integral{\R^{d-1}}{\max\left\{\frac{v(\alpha,\sigma,\tau)}{v(\infty,s,t)},\frac{v(\infty,s,t)}{v(\alpha,\sigma,\tau)}\right\}^q\abs{\ip{\psi_{(\infty,s,t)}}{\psi_{(\alpha,\sigma,\tau)}}}^q}{s}}{t}\\
= \integral{\R^d}{\integral{\R^{d-1}}{\max\left\{\frac{v(\alpha,\sigma,\tau)}{v(\infty,s,t)},\frac{v(\infty,s,t)}{v(\alpha,\sigma,\tau)}\right\}^q\abs{\ip{\Phi}{\psi_{(\alpha,\sigma-s,S_s^{-1}(\tau-t))}}}^q}{s}}{t}.
\end{align*}
Substituting $s' = \sigma-s$ and $t' = S_{\sigma-s'}^{-1}(\tau-t)$ then leads to
\begin{equation}\label{eq:firstint}
\begin{split}
\integral{\R^d}{\integral{\R^{d-1}}{\max\left\{\frac{v(\alpha,\sigma,\tau)}{v(\infty,s,t)},\frac{v(\infty,s,t)}{v(\alpha,\sigma,\tau)}\right\}^q\abs{\ip{\psi_{(\infty,s,t)}}{\psi_{(\alpha,\sigma,\tau)}}}^q}{s}}{t}\\
= \integral{\R^d}{\integral{\R^{d-1}}{\max\left\{\frac{v(\alpha,\sigma,\tau)}{v(\infty,\tilde{\sigma_1},t)},\frac{v(\infty,\tilde{\sigma_1},t)}{v(\alpha,\sigma,\tau)}\right\}^q\absbig{\ipbig{\Phi}{\psi_{(\alpha,s',S_{\sigma-s'}^{-1}(\tau-t))}}}^q}{s'}}{t}\\
= \integral{\R^d}{\integral{\R^{d-1}}{\max\left\{\frac{v(\alpha,\sigma,\tau)}{v(\infty,\tilde{\sigma_1},\tilde{\tau_1})},\frac{v(\infty,\tilde{\sigma_1},\tilde{\tau_1})}{v(\alpha,\sigma,\tau)}\right\}^q\abs{\ip{\Phi}{\psi_{(\alpha,s',t')}}}^q}{s'}}{t'}.
\end{split}
\end{equation}
Analogously we see that
\begin{equation}\label{eq:secondint}
\begin{split}
\integral{\R^d}{\integral{\R^{d-1}}{\int\limits_{-1}^{1} \max\left\{\frac{v(\infty,\sigma,\tau)}{v(a,s,t)},\frac{v(a,s,t)}{v(\infty,\sigma,\tau)}\right\}^q\abs{\ip{\psi_{(a,s,t)}}{\psi_{(\infty,\sigma,\tau)}}}^q\,\frac{\mathrm{d}a}{\abs{a}^{d+1}}}{s}}{t}\\
= \integral{\R^d}{\integral{\R^{d-1}}{\int\limits_{-1}^{1} \max\left\{\frac{v(\infty,\sigma,\tau)}{v(a,\tilde{\sigma_2},\tilde{\tau_2})},\frac{v(a,\tilde{\sigma_2},\tilde{\tau_2})}{v(\infty,\sigma,\tau)}\right\}^q\abs{\ip{\Phi}{\psi_{(a,s',t')}}}^q\,\frac{\mathrm{d}a}{\abs{a}^{d+1}}}{s'}}{t'},
\end{split}
\end{equation}
for $\sigma \in \R^{d-1}$ and $\tau \in \R^d$.
Now let $\alpha \in [-1,1]^*$.
Then~\prettyref{eq:IP4} yields
\begin{align*}
\lefteqn{\integral{\R^d}{\integral{\R^{d-1}}{\int\limits_{-1}^1 \max\left\{\frac{v(\alpha,\sigma,\tau)}{v(a,s,t)},\frac{v(a,s,t)}{v(\alpha,\sigma,\tau)}\right\}^q\abs{\ip{\psi_{(a,s,t)}}{\psi_{(\alpha,\sigma,\tau)}}}^q\,\frac{\mathrm{d}a}{\abs{a}^{d+1}}}{s}}{t}}\\
&= \int\limits_{\R^d} \int\limits_{\R^{d-1}} \int\limits_{-1}^1 \max\left\{\frac{v(\alpha,\sigma,\tau)}{v(a,s,t)},\frac{v(a,s,t)}{v(\alpha,\sigma,\tau)}\right\}^q\\
&\hspace{4cm}\cdot\absbig{\ipbig{\Psi}{\psi_{(a {\alpha}^{-1},\abs{\alpha}^{\frac{1}{d} - 1} (s - \sigma),A_{\alpha}^{-1} S_{\sigma}^{-1} (t - \tau))}}}^q\,\frac{\mathrm{d}a}{\abs{a}^{d+1}}\,\mathrm{d}s\,\mathrm{d}t,
\end{align*}
which---by substituting $a' := a {\alpha}^{-1}$---leads to
\begin{align*}
\lefteqn{\integral{\R^d}{\integral{\R^{d-1}}{\int\limits_{-1}^1 \max\left\{\frac{v(\alpha,\sigma,\tau)}{v(a,s,t)},\frac{v(a,s,t)}{v(\alpha,\sigma,\tau)}\right\}^q\abs{\ip{\psi_{(a,s,t)}}{\psi_{(\alpha,\sigma,\tau)}}}^q\,\frac{\mathrm{d}a}{\abs{a}^{d+1}}}{s}}{t}}\\
&= \int\limits_{\R^d}\int\limits_{\R^{d-1}}\int\limits_{-\abs{\alpha}^{-1}}^{\abs{\alpha}^{-1}} \max\left\{\frac{v(\alpha,\sigma,\tau)}{v(\tilde{\alpha},s,t)},\frac{v(\tilde{\alpha},s,t)}{v(\alpha,\sigma,\tau)}\right\}^q\\
&\hspace{4cm}\cdot\absbig{\ipbig{\Psi}{\psi_{(a',\abs{\alpha}^{\frac{1}{d} - 1} (s - \sigma),A_{\alpha}^{-1} S_{\sigma}^{-1} (t - \tau))}}}^q \frac{1}{\abs{\alpha}^d}\,\frac{\mathrm{d}a'}{\abs{a'}^{d+1}}\,\mathrm{d}s\,\mathrm{d}t.
\end{align*}
Again, substituting with $s' := \abs{\alpha}^{\frac{1}{d}-1} (s - \sigma)$ and $t' := A_{\alpha}^{-1} S_{\sigma}^{-1} (t - \tau)$, we get
\begin{equation}\label{eq:thirdint}
\begin{split}
\lefteqn{\integral{\R^d}{\integral{\R^{d-1}}{\int\limits_{-1}^1 \max\left\{\frac{v(\alpha,\sigma,\tau)}{v(a,s,t)},\frac{v(a,s,t)}{v(\alpha,\sigma,\tau)}\right\}^q\abs{\ip{\psi_{(a,s,t)}}{\psi_{(\alpha,\sigma,\tau)}}}^q\,\frac{\mathrm{d}a}{\abs{a}^{d+1}}}{s}}{t}}\\
&= \int\limits_{\R^d}\int\limits_{\R^{d-1}}\int\limits_{-\abs{\alpha}^{-1}}^{\abs{\alpha}^{-1}} \max\left\{\frac{v(\alpha,\sigma,\tau)}{v(\tilde{\alpha},\tilde{\sigma_3},t)},\frac{v(\tilde{\alpha},\tilde{\sigma_3},t)}{v(\alpha,\sigma,\tau)}\right\}^q\\
&\hspace{4cm}\cdot\absbig{\ipbig{\Psi}{\psi_{(a',s',A_{\alpha}^{-1} S_{\sigma}^{-1} (t - \tau))}}}^q \abs{\alpha}^{\frac{1}{d} - 2}\,\frac{\mathrm{d}a'}{\abs{a'}^{d+1}}\,\mathrm{d}s'\,\mathrm{d}t\\
&= \integral{\R^d}{\integral{\R^{d-1}}{\int\limits_{-\abs{\alpha}^{-1}}^{\abs{\alpha}^{-1}} \max\left\{\frac{v(\alpha,\sigma,\tau)}{v(\tilde{\alpha},\tilde{\sigma_3},\tilde{\tau_3})},\frac{v(\tilde{\alpha},\tilde{\sigma_3},\tilde{\tau_3})}{v(\alpha,\sigma,\tau)}\right\}^q\abs{\ip{\Psi}{\psi_{(a',s',t')}}}^q\,\frac{\mathrm{d}a'}{\abs{a'}^{d+1}}}{s'}}{t'}.
\end{split}
\end{equation}
Using~\prettyref{eq:firstint},~\prettyref{eq:secondint}, and~\prettyref{eq:thirdint}, we now have
\begin{align*}
\lefteqn{\esssup_{(\alpha,\sigma,\tau) \in X} \integral{X}{\abs{R_\frF ((\alpha,\sigma,\tau),(a,s,t))}^q m_v((\alpha,\sigma,\tau),(a,s,t))^q}{\mu (a,s,t)}}\\
&= \esssup_{(\alpha,\sigma,\tau) \in X} \int\limits_{\R^d}\int\limits_{\R^{d-1}} \Biggl(\max\left\{\frac{v(\alpha,\sigma,\tau)}{v(\infty,s,t)},\frac{v(\infty,s,t)}{v(\alpha,\sigma,\tau)}\right\}^q\abs{\ip{\psi_{(\infty,s,t)}}{\psi_{(\alpha,\sigma,\tau)}}}^q\\
&\hspace{1cm}+ \int\limits_{-1}^{1} \max\left\{\frac{v(\alpha,\sigma,\tau)}{v(a,s,t)},\frac{v(a,s,t)}{v(\alpha,\sigma,\tau)}\right\}^q\abs{\ip{\psi_{(a,s,t)}}{\psi_{(\alpha,\sigma,\tau)}}}^q\,\frac{\mathrm{d}a}{\abs{a}^{d+1}} \Biggr)\,\mathrm{d}s\,\mathrm{d}t\\
&= \max \Biggl\{ \esssup_{(\sigma,\tau) \in \R^{d-1}\times \R^d} \int\limits_{\R^d} \int\limits_{\R^{d-1}} \Biggl(\max\left\{\frac{v(\infty,\sigma,\tau)}{v(\infty,s,t)},\frac{v(\infty,s,t)}{v(\infty,\sigma,\tau)}\right\}^q\abs{\ip{\psi_{(\infty,s,t)}}{\psi_{(\infty,\sigma,\tau)}}}^q\\
&\hspace{1cm}+ \int\limits_{-1}^{1} \max\left\{\frac{v(\infty,\sigma,\tau)}{v(a,s,t)},\frac{v(a,s,t)}{v(\infty,\sigma,\tau)}\right\}^q\abs{\ip{\psi_{(a,s,t)}}{\psi_{(\infty,\sigma,\tau)}}}^q\,\frac{\mathrm{d}a}{\abs{a}^{d+1}} \Biggr)\,\mathrm{d}s\,\mathrm{d}t ,\\
& \esssup_{(\alpha,\sigma,\tau) \in [-1,1]^*\times \R^{d-1}\times \R^d} \int\limits_{\R^d} \int\limits_{\R^{d-1}} \Biggl(\max\left\{\frac{v(\alpha,\sigma,\tau)}{v(\infty,s,t)},\frac{v(\infty,s,t)}{v(\alpha,\sigma,\tau)}\right\}^q\abs{\ip{\psi_{(\infty,s,t)}}{\psi_{(\alpha,\sigma,\tau)}}}^q\\
&\hspace{1cm}+ \int\limits_{-1}^{1} \max\left\{\frac{v(\alpha,\sigma,\tau)}{v(a,s,t)},\frac{v(a,s,t)}{v(\alpha,\sigma,\tau)}\right\}^q\abs{\ip{\psi_{(a,s,t)}}{\psi_{(\alpha,\sigma,\tau)}}}^q\,\frac{\mathrm{d}a}{\abs{a}^{d+1}} \Biggr) \,\mathrm{d}s\,\mathrm{d}t \Biggr\}\\
&= \max \Biggl\{ \esssup_{(\sigma,\tau)\in \R^{d-1}\times \R^d} \int\limits_{\R^d}\int\limits_{\R^{d-1}}\biggl( \max\left\{\frac{v(\infty,\sigma,\tau)}{v(\infty,\tilde{\sigma_1},\tilde{\tau_1})},\frac{v(\infty,\tilde{\sigma_1},\tilde{\tau_1})}{v(\infty,\sigma,\tau)}\right\}^q \abs{\ip{\Phi}{\psi_{(\infty,s',t')}}}^q \\
&\hspace{1cm}+ \int\limits_{-1}^{1} \max\left\{\frac{v(\infty,\sigma,\tau)}{v(a,\tilde{\sigma_2},\tilde{\tau_2})},\frac{v(a,\tilde{\sigma_2},\tilde{\tau_2})}{v(\infty,\sigma,\tau)}\right\}^q\abs{\ip{\Phi}{\psi_{(a,s',t')}}}^q\,\frac{\mathrm{d}a}{\abs{a}^{d+1}} \biggr)\,\mathrm{d}s'\,\mathrm{d}t',\\
& \esssup_{(\alpha,\sigma,\tau) \in [-1,1]^*\times \R^{d-1}\times\R^d} \int\limits_{\R^d}\int\limits_{\R^{d-1}}\biggl(\max\left\{\frac{v(\alpha,\sigma,\tau)}{v(\infty,\tilde{\sigma_1},\tilde{\tau_1})},\frac{v(\infty,\tilde{\sigma_1},\tilde{\tau_1})}{v(\alpha,\sigma,\tau)}\right\}^q \abs{\ip{\Phi}{\psi_{(\alpha,s',t')}}}^q \\
&\hspace{1cm}+ \int\limits_{-\abs{\alpha}^{-1}}^{\abs{\alpha}^{-1}} \max\left\{\frac{v(\alpha,\sigma,\tau)}{v(\tilde{\alpha},\tilde{\sigma_3},\tilde{\tau_3})},\frac{v(\tilde{\alpha},\tilde{\sigma_3},\tilde{\tau_3})}{v(\alpha,\sigma,\tau)}\right\}^q\abs{\ip{\Psi}{\psi_{(a',s',t')}}}^q\,\frac{\mathrm{d}a'}{\abs{a'}^{d+1}} \biggr)\,\mathrm{d}s'\,\mathrm{d}t' \Biggr\}
\end{align*}
\end{proof}

We use~\prettyref{lemma:kernel} to prove $R_\frF \in \mathcal{A}_{q,m_v}$ for certain functions $\Phi$ and $\Psi$.
Since it is not possible to construct functions $\Phi, \Psi \in L_2$ with compact support in the spatial domain satisfying the conditions in~\prettyref{remark:choicePhi},
 in the following we assume $\Psi$ to be a \emph{bandlimited} Schwartz function, in particular \[ \supp \hat{\Psi} \subseteq ([-a_1,-a_0]\cup [a_0,a_1])\times Q_b \] with $0 < a_0 < a_1$ and $Q_b := \cartesian_{i = 1}^{d-1} [-b_i,b_i]$ for $b \in \R^{d-1}_+$.
The function $\Phi$ is chosen in the same way as in \prettyref{remark:choicePhi}.
It follows that
\begin{equation*}
\supp\hat\Phi \subseteq [-a_1,a_1] \cartesian Q_b.
\end{equation*}
As weight functions on $X$ we consider
\begin{equation}\label{eq:weight}
v_{r}(\alpha,s,t) = v_{r}(\alpha) := \begin{cases}
1, &\alpha = \infty,\\
|\alpha|^{-r}, &\alpha\in [-1,1]^\ast,
\end{cases}
\end{equation}
with $r\in\R_{\geq 0}$, which satisfy all necessary conditions.
Through simple calculations one can verify the following properties of the moderate weight $m_{v_r}$ associated with $v_{r}$ for $a,a' \in [-1,1]^\ast$:
\begin{align}
m_{v_{r}}(\infty,\infty) &= 1,\label{eq:propertyweight1}\\
m_{v_{r}}(a,\infty) &= m_{v_r}(\infty,a) = |a|^{-r},\label{eq:propertyweight2}\\
m_{v_{r}}(a,a') &= \max\left\{\frac{\abs{a}}{\abs{a'}},\frac{\abs{a'}}{\abs{a}}\right\}^{-r}.\label{eq:propertyweight3}
\end{align}

The following technical lemma concerns support properties of $\Phi$ and $\Psi$ in the frequency domain, similar to~Lemma~3.1,~Ref.~\cite{Dahlke:ShearletArbitrary}.

\begin{lemma}\label{lemma:param}
Let $0<a_0<a_1$ and $b\in\R^{d-1}_+$. Then with $\Psi$ and $\Phi$ defined as above and for $a\in\R^\ast$ and $s\in\R^{d-1}$ we have
\begin{thmlist}
\item\label{list:parampsipsi} $\hat{\Psi}\hat\Psi(A_aS_s^T\cdot)\not\equiv 0$ implies $a\in[-\frac{a_1}{a_0},-\frac{a_0}{a_1}]\cup[\frac{a_0}{a_1},\frac{a_1}{a_0}]$ and $s\in Q_{d_1}$ with $d_1:=(a_0^{-1}+a_0^{-(1+\frac{1}{d})}a_1^{\frac{1}{d}})b$,
\item\label{list:paramphipsi} Assume $\abs{a}\leq 1$ then $\hat{\Phi}\hat\Psi(A_aS_s^T\cdot)\not\equiv 0$ implies $a\in[-1,-\frac{a_0}{a_1}]\cup[\frac{a_0}{a_1},1]$ and $s\in Q_{d_2}$ with $d_2:=(a_0^{-1}+a_0^{-(1+\frac1d)}a_1^\frac1d)b$,
\item\label{list:paramphiphi} $\supp\hat{\Phi}\hat{\Phi}(S_s^T\cdot)\subseteq\Omega_s:=\{x\in\R^d:|x_1|\leq a_1,\max\{-b_i,-b_i-s_{i-1}x_1\}\leq x_i\leq\min\{b_i,b_i-s_{i-1}x_1\},i=2,\ldots d\}$.
\end{thmlist}
\end{lemma}

\begin{proof}
The proof of \prettyref{list:parampsipsi} can be found in~Lemma~3.1,~Ref.~\cite{Dahlke:ShearletArbitrary}.
To prove \prettyref{list:paramphipsi} we assume there exists a $\xi \in \supp \hat{\Phi}\cap \supp \hat{\Psi}(A_a S_s^T\cdot)$ which means that $\xi \in \supp \hat{\Phi} \mbox{ and } A_a S_s^T \xi \in \supp \hat{\Psi}.$
This leads to
\begin{gather}
\abs{\xi_1} \leq a_1, \label{eq:supp1}\\
-b_i \leq \xi_{i+1} \leq b_i, \label{eq:supp2}\\
a_0 \leq \abs{a} \abs{\xi_1} \leq a_1, \label{eq:supp3}\\
-b_i \abs{a}^{-\frac{1}{d}} - \xi_1 s_i \leq \xi_{i+1} \leq b_i \abs{a}^{-\frac{1}{d}} - \xi_1 s_i, \label{eq:supp4}
\end{gather}
for $i = 1,\ldots,d-1$.
By \eqref{eq:supp1} and \eqref{eq:supp3} it follows that $\abs{a} \geq \frac{a_0}{a_1}$ which means $a \in [-1,-\frac{a_0}{a_1}]\cup[\frac{a_0}{a_1},1]$.
Using \eqref{eq:supp3} and $\abs{a} \leq 1$, it follows that $a_0 \leq \abs{a} \abs{\xi_1} \leq \abs{\xi_1}$.
Also, with \eqref{eq:supp4} and \eqref{eq:supp2} we obtain
\begin{equation*}
-b_i \abs{a}^{-\frac{1}{d}}-b_i \leq \xi_1 s_i \leq b_i \abs{a}^{-\frac{1}{d}}+b_i,
\end{equation*}
which leads to
\begin{equation*}
\abs{s_i} \leq \abs{\xi_1}^{-1}( b_i \abs{a}^{-\frac{1}{d}} + b_i) \leq a_0^{-1} b_i \left(\frac{a_0}{a_1}\right)^{-\frac{1}{d}} + a_0^{-1}b_i
\end{equation*}
for $i = 1,\ldots,d-1$ which proves \prettyref{list:paramphipsi}.
To prove \prettyref{list:paramphiphi} we assume there exists $\xi\in\supp\hat\Phi\cap\supp\hat\Phi(S_s^T\cdot)$ which means that $\xi\in\supp\hat\Phi$ and $S_s^T\xi\in\supp\hat\Phi$. This leads to
\begin{gather}
\abs{\xi_1} \leq a_1,\\
-b_i \leq \xi_{i+1} \leq b_i,\\
-b_i \leq \xi_1 s_i + \xi_{i+1} \leq b_i
\end{gather}
for all $i=1,\ldots,d-1$, which means $\xi\in\Omega_s$.
\end{proof}

The following two auxiliary Lemmas are of technical nature only and the proof of \prettyref{lemma:subconvolution1} is based on a draft by Steidl, Dahlke, H\"auser and Teschke.

\begin{lemma}\label{lemma:subconvolution1}
For all $y,z\in\R$, $\lambda,\lambda'>0$ and $k>1$ the following integral estimation holds true
\begin{align*}
\integral{\R}{(1+\lambda|x-y|)^{-k}(1+\lambda'|x-z|)^{-k}}{x} \lesssim \max\{\lambda,\lambda'\}^{-1}(1+\min\{\lambda,\lambda'\}|y-z|)^{-k}.
\end{align*}
\end{lemma}

\begin{proof}
Let $y,z\in\R$ be arbitrary and assume without loss of generality that $\lambda\leq\lambda'$. Assume further that $|y-z|\leq\lambda^{-1}$, then
\begin{equation*}
(1+\lambda|x-y|)^{-k} \leq 1 \leq 2^k (1+\lambda|y-z|)^{-k}
\end{equation*}
and thus
\begin{align}\label{eq:lemma37_1}
\integral{\R}{(1+\lambda|x-y|)^{-k}(1+\lambda'|x-z|)^{-k}}{x} &\lesssim (1+\lambda|y-z|)^{-k}\integral{\R}{(1+\lambda'|x-z|)^{-k}}{x} \notag\\
&= (1+\lambda'|y-z|)^{-k}\frac{1}{\lambda'}\integral{\R}{(1+|x|)^{-k}}{x} \notag\\
&\lesssim \frac{1}{\lambda'}(1+\lambda|y-z|)^{-k}.
\end{align}
On the other hand if $|y-z|>\lambda^{-1}$ let $H_y$ and $H_z$ be the two half-axes containing the points $y$ and $z$ respectively, such that $H_y\cap H_z=\{\frac{y+z}{2}\}$. Then, for every $x\in H_z$ it holds $|x-y|\geq\frac12|y-z|$ and thus
\begin{align}\label{eq:lemma37_2}
\integral{H_z}{(1+\lambda|x-y|)^{-k}(1+\lambda'|x-z|)^{-k}}{x} &\leq \left(1+\frac{\lambda}{2}|y-z|\right)^{-k}\integral{H_z}{(1+\lambda'|x-z|)^{-k}}{x} \notag\\
&\lesssim (1+\lambda|y-z|)^{-k}\frac{1}{\lambda'}\integral{\R}{(1+|x|)^{-k}}{x} \notag\\
&\lesssim \frac{1}{\lambda'}(1+\lambda|y-z|)^{-k}
\end{align}
Similarily for every $x\in H_y$ it holds $|x-z|\geq\frac12|y-z|$ and since $|y-z|>\lambda^{-1}$ we first deduce
\begin{align*}
(1+\lambda'|x-z|)^{-k} &\leq \left(\frac{\lambda'}{2}|y-z|\right)^{-k} \lesssim \left(\frac{\lambda}{\lambda'}\right)^k(\lambda|y-z|)^{-k}\\
& \lesssim \left(\frac{\lambda}{\lambda'}\right)^k(1+\lambda|y-z|)^{-k} \leq \frac{\lambda}{\lambda'}(1+\lambda|y-z|)^{-k}
\end{align*}
and hence we derive the estimate
\begin{align}\label{eq:lemma37_3}
\integral{H_y}{(1+\lambda|x-y|)^{-k}(1+\lambda'|x-z|)^{-k}}{x} &\leq \frac{\lambda}{\lambda'}(1+\lambda|y-z|)^{-k}\integral{H_y}{(1+\lambda|x-y|)^{-k}}{x} \notag\\
&\leq \frac{1}{\lambda'}(1+\lambda|y-z|)^{-k}\integral{\R}{(1+|x|)^{-k}}{x} \notag\\
&\lesssim \frac{1}{\lambda'}(1+\lambda|y-z|)^{-k}.
\end{align}
Combining \prettyref{eq:lemma37_2} and \prettyref{eq:lemma37_3} thus yields
\begin{align*}
&\integral{\R}{(1+|x-y|)^{-k}(1+\lambda|x-z|)^{-k}}{x}\\
&\qquad= \biggl(\int\limits_{H_y}+\int\limits_{H_z}\biggr)(1+|x-y|)^{-k}(1+\lambda|x-z|)^{-k}\,\mathrm{d}x \lesssim \frac{1}{\lambda'}(1+\lambda|y-z|)^{-k}
\end{align*}
and together with \prettyref{eq:lemma37_1} this completes the proof.
\end{proof}

\begin{lemma}\label{lemma:subconvolution2}
For all $y,z\in\R^\ast$, $\lambda\neq 0$ and $k>1$ we have
\begin{align*}
&\integral{\R}{(1+|x|)^{-k}(1+|x-y|)^{-k}(1+|\lambda x-z|)^{-k}}{x} \\
&\hspace{1cm} \lesssim (1+|y|)^{-k}\max\{1,|\lambda|\}^{-1}\\
&\hspace{3cm}\cdot\left[\left(1+\min\{1,|\lambda|\}\left|y-\frac{z}{\lambda}\right|\right)^{-k}+\left(1+\min\{1,|\lambda|\}\left|\frac{z}{\lambda}\right|\right)^{-k}\right].
\end{align*}
\end{lemma}

\begin{proof}
We use the ideas of the proof of~Lemma~11.1.1,~Ref.~\cite{Groechenig}, as well as \prettyref{lemma:subconvolution1} and define the set $N_y:=\{x\in\R:|x-y|\leq\frac{|y|}{2}\}$. For all $x\in N_y$ it follows that $\abs{x}\geq\frac{\abs{y}}{2}$ and thus
\begin{equation*}
(1+|x|)^{-k} \leq \left(1+\frac{|y|}{2}\right)^{-k} \leq 2^k(1+|y|)^{-k}.
\end{equation*}
On the other hand if $x\in N_y^c$ one has $(1+|x-y|)^{-k}\leq(1+\frac{|y|}{2})^{-k}$. Hence, with \prettyref{lemma:subconvolution1} we can derive
\begin{align*}
&\integral{\R}{(1+|x|)^{-k}(1+|x-y|)^{-k}(1+|\lambda x-z|)^{-k}}{x} \\
&\hspace{1cm} = \biggl(\int\limits_{N_y}+\int\limits_{N_y^c}\biggr)(1+|x|)^{-k}(1+|x-y|)^{-k}(1+|\lambda x-z|)^{-k}\,\mathrm{d}x \\
&\hspace{1cm} \lesssim (1+|y|)^{-k}\integral{\R}{(1+|x-y|)^{-k}\left(1+|\lambda|\left|x-\frac{z}{\lambda}\right|\right)^{-k}}{x} \\
&\hspace{3cm}+ (1+|y|)^{-k}\integral{\R}{(1+|x|)^{-k}\left(1+|\lambda|\left|x-\frac{z}{\lambda}\right|\right)^{-k}}{x} \\
&\hspace{1cm} \lesssim (1+|y|)^{-k}\max\{1,|\lambda|\}^{-1}\left(1+\min\{1,|\lambda|\}\left|y-\frac{z}{\lambda}\right|\right)^{-k} \\
&\hspace{3cm}+(1+|y|)^{-k}\max\{1,|\lambda|\}^{-1}\left(1+\min\{1,|\lambda|\}\left|\frac{z}{\lambda}\right|\right)^{-k},
\end{align*}
which concludes the proof.
\end{proof}

Now we are able to prove that the integrability condition on the kernel function is satisfied, i.e. that $R_\frF \in \cA_{q,{m_{v_{r}}}}$.

\begin{theorem}\label{theorem:frameinaq}
Let $\Psi \in L_1(\R^d)\cap L_2(\R^d)$ be an admissible shearlet with \[ \supp \hat{\Psi} \subseteq ([-a_1,-a_0]\cup [a_0,a_1])\times Q_b. \] Let $\Phi \in L_1(\R^d)\cap L_2(\R^d)$ be chosen as in \prettyref{remark:choicePhi} so that condition \prettyref{eq:assumptions} is satisfied for $0 < a_0 < a_1$ and $b \in \R_+^{d-1}$ and additionally $\hat\Phi\in\mathscr{C}^\infty_0(\R^d)$.
Then, for every $q>1$ the kernel $R_\frF$ fulfills \[ R_\frF \in \mathcal{A}_{q,{m_{v_r}}}. \] 
\end{theorem}

\begin{proof}
For $q>1$ fixed we use \prettyref{lemma:kernel} and look at the four summands in \prettyref{eq:kernel} independently.
We need to show that all summands are bounded and for that we use \prettyref{lemma:param}.
Let $\tilde\alpha:=\alpha a$ and by using \prettyref{lemma:param} \prettyref{list:parampsipsi} with the specific weight $v_r$ we obtain
\begin{align}\label{eq:frameinaq13}
&\esssup_{(\alpha,\sigma,\tau)\in X}\int\limits_X\abs{R_\frF((\alpha,\sigma,\tau),(a,s,t))}^q m_{v_r}(\alpha,a)^q\,\mathrm{d}\mu(a,s,t) \notag\\
&\hspace{.5cm} =\max\Biggl\{ \int\limits_{\R^{d}}\int\limits_{\R^{d-1}}\biggl(\abs{\ip{\Phi}{\psi_{(\infty,s,t)}}}^q+\int\limits_{-1}^1\abs{a}^{-rq}\abs{\ip{\Phi}{\psi_{(a,s,t)}}}^q\,\frac{\mathrm{d}a}{\abs{a}^{d+1}}\biggr)\,\mathrm{d}s\,\mathrm{d}t, \notag\\
&\hspace{2.5cm} \esssup_{\alpha\in[-1,1]^\ast}\int\limits_{\R^d}\int\limits_{\R^{d-1}}\biggl(\abs{\alpha}^{-rq}\abs{\ip{\Phi}{\psi_{(\alpha,s,t)}}}^q \\
&\hspace{3.5cm} +\int\limits_{-\abs{\alpha}^{-1}}^{\abs{\alpha}^{-1}}\max\left\{\frac{\abs{\alpha}}{\abs{\tilde\alpha}},\frac{\abs{\tilde\alpha}}{\abs{\alpha}}\right\}^{-rq}\abs{\ip{\Psi}{\psi_{(a,s,t)}}}^q\,\frac{\mathrm{d}a}{\abs{a}^{d+1}}\biggr)\,\mathrm{d}s\,\mathrm{d}t\Biggr\}. \notag
\end{align}
We need to show that all four summands of \prettyref{eq:frameinaq13} are bounded and for this we will treat the summands independently.\par 
First, since $\cF(f^\ast)=\overline{\cF(f)}$ and \mbox{$\Phi\ast\psi_{(a,s,0)}^\ast\in L_1(\R^d)$} we obtain
\begin{align*}
\ip{\Phi}{\psi_{(a,s,t)}} = (\Phi\ast\psi_{(a,s,0)}^\ast)(t) = \cF^{-1}(\cF(\Phi\ast\psi_{(a,s,0)}^\ast))(t) = \cF^{-1}(\hat\Phi\overline{\cF(\psi_{(a,s,0)})})(t)
\end{align*}
which leads to
\begin{align*}
\int\limits_{R^d}\abs{\ip{\Phi}{\psi_{(a,s,t)}}}^q\,\mathrm{d}t = \norm{\cF^{-1}(\hat\Phi\overline{\cF(\psi_{(a,s,0)})})}{L_q}^q.
\end{align*}
Applying \prettyref{lemma:param} \prettyref{list:paramphipsi} we see that $\hat\Phi\cF(\psi_{(a,s,0)})\equiv 0$ for all $s\notin Q_{d_2}$ or $a\notin[-1,-\frac{a_0}{a_1}]\cup[\frac{a_0}{a_1},1]$, which implies \[ \norm{\cF^{-1}(\hat\Phi\overline{\cF(\psi_{(a,s,0)})})}{L_q}^q=0 \] for all $s\notin Q_{d_2}$ or $a\notin[-1,-\frac{a_0}{a_1}]\cup[\frac{a_0}{a_1},1]$. Thus, with \prettyref{lemma:param} \prettyref{list:paramphipsi} we derive
\begin{align}\label{eq:frameinaq1}
&\esssup_{\alpha\in[-1,1]^\ast}\int\limits_{\R^d}\int\limits_{\R^{d-1}}\abs{\alpha}^{-rq}\abs{\ip{\Phi}{\psi_{(\alpha,s,t)}}}^q\,\mathrm{d}s\,\mathrm{d}t \notag\\
&\hspace{2cm} =\esssup_{\alpha\in[-1,1]^\ast}\abs{\alpha}^{-rq}\int\limits_{\R^{d-1}}\norm{\cF^{-1}(\hat\Phi\overline{\cF(\psi_{(\alpha,s,0)})})}{L_q}^q\,\mathrm{d}s \notag\\
&\hspace{2cm} =\esssup_{\alpha\in[-1,-\frac{a_0}{a_1}]\cup[\frac{a_0}{a_1},1]}\abs{\alpha}^{-rq}\int\limits_{Q_{d_2}}\norm{\Phi\ast\psi_{(\alpha,s,0)}^\ast}{L_q}^q\,\mathrm{d}s < \infty.
\end{align}
Using the same arguments as well as \prettyref{lemma:param} \prettyref{list:parampsipsi} we obtain
\begin{align}\label{eq:frameinaq2}
&\esssup_{\alpha\in[-1,1]^\ast}\int\limits_{\R^d}\int\limits_{\R^{d-1}}\int\limits_{-\abs{\alpha}^{-1}}^{\abs{\alpha}^{-1}}\max\left\{\frac{\abs{\alpha}}{\abs{\tilde\alpha}},\frac{\abs{\tilde\alpha}}{\abs{\alpha}}\right\}^{-rq}\abs{\ip{\Psi}{\psi_{(a,s,t)}}}^q\,\frac{\mathrm{d}a}{\abs{a}^{d+1}}\,\mathrm{d}s\,\mathrm{d}t \notag\\
& \leq \int\limits_{\R}\max\left\{\abs{a},\abs{a}^{-1}\right\}^{-rq}\int\limits_{\R^{d-1}}\norm{\cF^{-1}(\hat\Phi\overline{\cF(\psi_{(a,s,0)})})}{L_q}^q\,\mathrm{d}s\,\frac{\mathrm{d}a}{\abs{a}^{d+1}} \notag\\
& = \biggl(\int\limits_{-\frac{a_1}{a_0}}^{-\frac{a_0}{a_1}}+\int\limits_{\frac{a_0}{a_1}}^{\frac{a_1}{a_0}}\biggr)\max\left\{\abs{a},\abs{a}^{-1}\right\}^{-rq}\int\limits_{Q_{d_1}}\norm{\Phi\ast\psi_{(a,s,0)}^\ast}{L_q}^q\,\mathrm{d}s\frac{\mathrm{d}a}{\abs{a}^{d+1}} < \infty.
\end{align}
Again, with analogous arguments and \prettyref{lemma:param} \prettyref{list:paramphipsi} it follows that
\begin{align}\label{eq:frameinaq3}
&\int\limits_{\R^{d}}\int\limits_{\R^{d-1}}\int\limits_{-1}^1\abs{a}^{-rq}\abs{\ip{\Phi}{\psi_{(a,s,t)}}}^q\,\frac{\mathrm{d}a}{\abs{a}^{d+1}}\,\mathrm{d}s\,\mathrm{d}t \notag\\
&\hspace{2cm} = \int\limits_{-1}^1\abs{a}^{-rq}\int\limits_{\R^{d-1}}\norm{\cF^{-1}(\hat\Psi\overline{\cF(\psi_{(a,s,0)})})}{L_q}^q\,\mathrm{d}s\,\frac{\mathrm{d}a}{\abs{a}^{d+1}} \notag\\
&\hspace{2cm} = \biggl(\int\limits_{-1}^{-\frac{a_0}{a_1}}+\int\limits_{\frac{a_0}{a_1}}^1\biggr)\abs{a}^{-rq}\int\limits_{Q_{d_2}}\norm{\Psi\ast\psi_{(a,s,0)}^\ast}{L_q}^q\,\mathrm{d}s\,\frac{\mathrm{d}a}{\abs{a}^{d+1}} < \infty.
\end{align}
For the last summand in \prettyref{eq:frameinaq13} we choose $q_0$, $q_1$ positive, such that $q_0+q_1=q$. We will specify the choice at the end of the proof. Then, it follows that
\begin{align}\label{eq:frameinaq4}
&\int\limits_{\R^{d}}\int\limits_{\R^{d-1}}\abs{\ip{\Phi}{\psi_{(\infty,s,t)}}}^q\,\mathrm{d}s\,\mathrm{d}t \notag\\
&\hspace{1cm} = \int\limits_{\R^{d-1}}\int\limits_{\R^d}\abs{(\Phi\ast\psi^\ast_{(\infty,s,0)}(t)}^{q_0+q_1}\,\mathrm{d}t\,\mathrm{d}s \notag\\
&\hspace{1cm} = \int\limits_{\R^{d-1}}\int\limits_{\R^d}\abs{(\Phi\ast\psi^\ast_{(\infty,s,0)})(t)}^{q_0}\abs{\mathcal{F}^{-1}(\hat\Phi\overline{\mathcal{F}(\psi_{(\infty,s,0)})})(t)}^{q_1}\,\mathrm{d}t\,\mathrm{d}s \notag\\
&\hspace{1cm} \lesssim \int\limits_{\R^{d-1}}\int\limits_{\R^d}\biggl(\int\limits_{\R^d}\abs{\Phi(x)\psi^\ast_{(\infty,s,0)}(x-t)}\,\mathrm{d}x\biggr)^{q_0}\,\mathrm{d}t\,\biggl(\int\limits_{\R^d}\abs{\hat\Phi(\omega)\mathcal{F}\psi^\ast_{(\infty,s,0)}(\omega)}\,\mathrm{d}\omega\biggr)^{q_1}\,\mathrm{d}s \notag\\
&\hspace{1cm} =: \int\limits_{\R^{d-1}}I_0(s) I_1(s)\,\mathrm{d}s.
\end{align}
In the following we will treat both factors $I_0$ and $I_1$ independently.
\par
\underline{$I_0(s)$:} We assume in the following $0<q_0<1$. Since $\hat\Phi\in\mathscr{C}_c^\infty(\R^d)$, for every $k\in\N$ it follows that $|\Phi(x)|\lesssim(1+|x|)^{-k}$ for all $x\in\R^d$ with the constant depending on $k$ and $d$. Then, 
\begin{align*}
I_0(s) &\lesssim \int\limits_{\R^d}\biggl(\int\limits_{\R^d}\prod_{i=1}^d\left[(1+|x_i+t_i|)^{-k}(1+|(S_{-s}x)_i|)^{-k}\right]\,\mathrm{d}x\biggr)\,\mathrm{d}t =: \int\limits_{\R^3} I_s(t)^{q_0} \,\mathrm{d}t
\end{align*}
for $s\in\R^{d-1}$ fixed and where $(S_{-s}x)_i$ denotes the $i$-th entry of the vector $S_{-s}x\in\R^d$. With this notation we intend to show
\begin{align}\label{eq:frameinaq5}
\int\limits_{\R^3}I_s(t)^{q_0}\,\mathrm{d}t \lesssim (1+\|s\|)^{1-q_0}\int\limits_{\R^d}\prod_{i=1}^d(1+|t_i|)^{-kq_0}\,\mathrm{d}t
\end{align}
with the constant depending on $k$ and $q_0$ only. For this we first show an auxiliary result for $d=3$ which we will then generalize to arbitrary dimensions. To illustrate our method we differentiate between the following four cases for $s\in\R^2$ with $s_1,s_2\neq 0$.
\par 
\underline{Case 1:} $|s_1|,|s_2|\leq 1$. With \prettyref{lemma:subconvolution1} and \prettyref{lemma:subconvolution2} we obtain
\begin{align}\label{eq:frameinaq6}
I_s(t) &\lesssim \int\limits_{\R^2}(1+|t_1+s_1x_2+s_2x_3|)^{-k}(1+|x_2+t_2|)^{-k}\notag\\
&\qquad\qquad\cdot(1+|x_2|)^{-k}(1+|x_3+t_3|)^{-k}(1+|x_3|)^{-k}\,\mathrm{d}(x_2,x_3) \notag\\
&\lesssim\int\limits_{\R}(1+|t_2|)^{-k}(1+|-s_1t_2+t_1+s_2x_3|)^{-k}(1+|x_3+t_3|)^{-k}(1+|x_3|)^{-k}\,\mathrm{d}x_3 \notag\\
&\hspace{1cm} + \int\limits_{\R}(1+|t_2|)^{-k}(1+|t_1+s_2x_3|)^{-k}(1+|x_3+t_3|)^{-k}(1+|x_3|)^{-k}\,\mathrm{d}x_3 \notag\\
&\lesssim (1+|t_2|)^{-k}(1+|t_3|)^{-k}\big[(1+|s_2t_3+s_1t_2+t_1|)^{-k} \notag\\
&\hspace{1cm}+(1+|s_1t_2-t_1|)^{-k}+(1+|s_2t_3+t_1|)^{-k}+(1+|t_1|)^{-k}\big].
\end{align}
\par 
\underline{Case 2:} $|s_1|\leq 1,|s_2|>1$. Again, with \prettyref{lemma:subconvolution1} and \prettyref{lemma:subconvolution2} we obtain
\begin{align}\label{eq:frameinaq7}
I_s(t) &\lesssim \int\limits_{\R}(1+|t_2|)^{-k}(1+|-s_1t_2+t_1+s_2x_3|)^{-k}(1+|x_3+t_3|)^{-k}(1+|x_3|)^{-k}\,\mathrm{d}x_3 \notag\\
&\hspace{1cm} + \int\limits_{\R}(1+|t_2|)^{-k}(1+|t_1+s_2x_3|)^{-k}(1+|x_3+t_3|)^{-k}(1+|x_3|)^{-k}\,\mathrm{d}x_3 \notag\\
&\lesssim |s_2|^{-1}(1+|t_2|)^{-k}(1+|t_3|)^{-k}\big[(1+|-t_3+s_1s_2^{-1}t_2-s_2^{-1}t_1|)^{-k} \\
&\hspace{1cm} +(1+|s_1s_2^{-1}t_2-s_2^{-1}t_1|)^{-k}+(1+|t_3+s_2^{-1}t_1|)^{-k}+(1+|s_2^{-1}t_1|)^{-k}\big]. \notag
\end{align}
\par 
\underline{Case 3:} $|s_1|>1,|s_2|\leq|s_1|$. Similarily we apply \prettyref{lemma:subconvolution1} and \prettyref{lemma:subconvolution2} to derive
\begin{align}\label{eq:frameinaq8}
I_s(t) &\lesssim \int\limits_{\R^2}(1+|t_1+s_1x_2+s_2x_3|)^{-k}(1+|x_2+t_2|)^{-k}\notag\\
&\qquad\qquad\qquad\qquad\cdot(1+|x_2|)^{-k}(1+|x_3+t_3|)^{-k}(1+|x_3|)^{-k}\,\mathrm{d}(x_2,x_3) \notag\\
&\lesssim \int\limits_{\R}(1+|t_2|)^{-k}|s_1|^{-1}(1+|-t_2+s_1^{-1}t_1+s_1^{-1}s_2x_3|)^{-k}\notag\\
&\qquad\qquad\qquad\qquad\cdot(1+|x_3+t_3|)^{-k}(1+|x_3|)^{-k}\,\mathrm{d}x_3 \notag\\
&\hspace{-.5cm} +\int\limits_{\R}(1+|t_2|)^{-k}|s_1|^{-1}(1+|s_1^{-1}t_1+s_1^{-1}s_2x_3|)^{-k}(1+|x_3+t_3|)^{-k}(1+|x_3|)^{-k}\,\mathrm{d}x_3 \notag\\
&\lesssim |s_1|^{-1}(1+|t_2|)^{-k}(1+|t_3|)^{-k}\big[(1+|-s_1^{-1}s_2t_3-t_2+s_1^{-1}t_1|)^{-k} \notag\\
&\hspace{.5cm} +(1+|-t_2+s_1^{-1}t_1|)^{-k}+(1+|-s_1^{-1}s_2t_3+s_1^{-1}t_1|)^{-k}+(1+|s_1^{-1}t_1|)^{-k}\big]. 
\end{align}
\par 
\underline{Case 4:} $|s_1|>1,|s_2|>|s_1|$. Finally we apply \prettyref{lemma:subconvolution1} and \prettyref{lemma:subconvolution2} again and conclude
\begin{align}\label{eq:frameinaq9}
I_s(t) &\lesssim \int\limits_{\R}(1+|t_2|)^{-k}|s_1|^{-1}(1+|-t_2+s_1^{-1}t_1+s_1^{-1}s_2x_3|)^{-k}\notag\\
&\qquad\qquad\qquad\qquad\qquad\qquad\cdot(1+|x_3+t_3|)^{-k}(1+|x_3|)^{-k}\,\mathrm{d}x_3 \notag\\
&\hspace{-.5cm} +\int\limits_{\R}(1+|t_2|)^{-k}|s_1|^{-1}(1+|s_1^{-1}t_1+s_1^{-1}s_2x_3|)^{-k}(1+|x_3+t_3|)^{-k}(1+|x_3|)^{-k}\,\mathrm{d}x_3 \notag\\
&\lesssim |s_2|^{-1}(1+|t_2|)^{-k}(1+|t_3|)^{-k}\big[(1+|-t_3-s_1s_2^{-1}t_2+s_2^{-1}t_1|)^{-k} \\
&\hspace{.5cm} +(1+|-s_1s_2^{-1}t_2+s_2^{-1}t_1|)^{-k}+(1+|-t_3+s_2^{-1}t_1|)^{-k}+(1+|s_2^{-1}t_1|)^{-k}\big].  \notag
\end{align}
\par 
The four cases \prettyref{eq:frameinaq6}, \prettyref{eq:frameinaq7}, \prettyref{eq:frameinaq8}, \prettyref{eq:frameinaq9} yield the estimate
\begin{equation}\label{eq:induktionsanfang}
I_s(t) \lesssim \abs{\det A_s^i}\sum_{i=1}^4\prod_{j=1}^3(1+\abs{(A_s^i t)_j})^{-k}
\end{equation}
with the Matrices $A_s^i$, $s\in\R^2$, $i=1,\ldots,4$, being of the form
\begin{align*}
A_s^i =
\begin{pmatrix}
\lambda & \mu & \nu \\
0 & 1 & 0 \\
0 & 0 & 1
\end{pmatrix}
\quad \mbox{for some } \lambda,\mu,\nu\in\R \mbox{ depending on } s_1,s_2.
\end{align*}
In particular it follows from the four cases that
\begin{align*}
\abs{\det A_s^i} = |\lambda| = \left\{
\begin{array}{ll}
1, & |s_1|,|s_2|\leq 1, \\
|s_2|^{-1}, & |s_1|\leq 1,|s_2|>1, \\
|s_1|^{-1}, & |s_1|>1,|s_2|\leq|s_1|, \\
|s_2|^{-1}, & |s_1|>1,|s_2|>|s_1|
\end{array}
\right\} = \max\{1,|s_1|,|s_2|\}^{-1}.
\end{align*}
\par
We now intend to show, that this result holds for arbitrary dimension. To this extend we fix $d\geq 3$ as well as $s\in\R^{d-1}$ with $s_i\neq 0$ for all $i=1,\ldots,d-1$ and assume that there exist matrices $A_s^i$ for $1\leq i\leq 2^{d-1}$ of the form
\begin{align}\label{eq:matrixform}
A_s^i=
\begin{pmatrix}
\ast & \ast & \cdots & \ast \\
 & 1 & & \\
 & & \ddots & \\
 & & & 1
\end{pmatrix}
\end{align}
with $\det A_s^i = (A_s^i)_{11}=\max\{1,|s_1|,\ldots,|s_{d-1}|\}^{-1}=\min\{1,|s_1|^{-1},\ldots,|s_{d-1}|^{-1}\}=:\min(s)$. Assume the estimate
\begin{align}\label{eq:induktionsvoraussetzung}
I_s(t) \lesssim \min(s)\sum_{i=1}^{2^{d-1}}\prod_{j=1}^d(1+|(A^i_s t)_j|)^{-k}
\end{align}
holds true for fixed $d$. As shown in \eqref{eq:induktionsanfang}, this readily is the case for $d=3$. We now intend to show that the estimate \eqref{eq:induktionsvoraussetzung} also holds for $d+1$. Then, \eqref{eq:induktionsvoraussetzung} will hold for arbitrary dimension by full induction over the dimension. To this end we fix $s\in\R^d$ with $s_i\neq 0$ for all $i=1,\ldots,d$ and define $\tilde x:=(x_1,\ldots,x_d)$, $\tilde s:=(s_1,\ldots,s_{d-1})$, $\tilde t:=(t_1,\ldots,t_d)$ and $u:=(t_1+s_d x_{d+1},t_2,\ldots,t_d)$. Then we deduce from \eqref{eq:induktionsvoraussetzung} the estimate
\begin{align}
I_s(t) &= \int\limits_{\R^{d+1}}\prod_{i=1}^{d+1}\left[(1+|x_i+t_i|)^{-k}(1+|(S_{-s}x)_i|)^{-k}\right]\,\mathrm{d}x \notag\\
&= \int\limits_\R(1+|x_{d+1}+t_{d+1}|)^{-k}(1+|x_{d+1}|)^{-k}\notag\\
&\hspace{2cm}\biggl(\int\limits_{R^d}\prod_{i=1}^d\left[(1+|x_i+u_i|)^{-k}(1+|(S_{-\tilde s}\tilde{x})_i|)^{-k}\right]\,\mathrm{d}\tilde{x}\biggr)\,\mathrm{d}x_{d+1} \notag\\
&= \int\limits_\R I_{\tilde{s}}(u)(1+|x_{d+1}+t_{d+1}|)^{-k}(1+|x_{d+1}|)^{-k}\,\mathrm{d}x_{d+1} \notag\\
&\lesssim \min(\tilde{s})\sum_{i=1}^{2^{d-1}}\prod_{j=1}^d\int\limits_\R (1+|(A^i_{\tilde{s}}u)_j|)^{-k}(1+|x_{d+1}+t_{d+1}|)^{-k}(1+|x_{d+1}|)^{-k}\,\mathrm{d}x_{d+1}, \label{eq:induktion1}
\end{align}
whereby we remember $(S_{-s}x)_i=x_i$ for all $i=2,\ldots,d+1$ and $(A^i_{\tilde{s}}u)_j=u_j$ for all $j=2,\ldots,d$. Since all integrals for $j\neq 1$ will remain unchanged we are now interested in the integrals in \eqref{eq:induktion1} for arbitrary $1\leq i\leq 2^{d-1}$, $j=1$ and obtain with \prettyref{lemma:subconvolution2}
\begin{align*}
&\int\limits_\R (1+|(A^i_{\tilde{s}}u)_1|)^{-k}(1+|x_{d+1}+t_{d+1}|)^{-k}(1+|x_{d+1}|)^{-k}\,\mathrm{d}x_{d+1} \\
&= \int\limits_\R (1+|(A^i_{\tilde{s}}\tilde{t})_1+\min(\tilde{s})s_d x_{d+1}|)^{-k}(1+|x_{d+1}+t_{d+1}|)^{-k}(1+|x_{d+1}|)^{-k}\,\mathrm{d}x_{d+1} \\
&\lesssim \min(\tilde{s})(1+|t_{d+1}|)^{-k}\max\{1,|\min(\tilde{s})s_d|\}^{-1} \\
&\hspace{2cm}\times\biggl[ \left( 1+ \left| \min\{1,|\min(\tilde{s})s_d|\}t_{d+1} - \frac{\min\{1,|\min(\tilde{s})s_d|\}}{|\min(\tilde{s})s_d|}(A_{\tilde{s}}^i \tilde{t})_1\right|\right)^{-k} \\
&\hspace{4cm} +\left(1+\frac{\min\{1,|\min(\tilde{s})s_d|\}}{|\min(\tilde{s})s_d|}|(A_{\tilde{s}}^i \tilde{t})_1|\right)^{-k}\,\biggr] \\
&= \max\{\min(\tilde{s})^{-1},|s_d|\}^{-1}\\
&\hspace{1cm}\cdot\left[(1+|(B^i_s t)_{d+1}|)^{-k}(1+|(B^i_s t)_1|)^{-k}+(1+|(C^i_s t)_{d+1}|)^{-k}(1+|(C^i_s t)_1|)^{-k}\right]
\end{align*}
for some matrices $B_s^i$, $C_s^i$ of the form \eqref{eq:matrixform} where
\begin{equation*}
(B_s^i)_{11}=(C_s^i)_{11}=(A_{\tilde{s}}^i)_{11}\left(\frac{\min\{1,|\min(\tilde{s})s_d|\}}{|\min(\tilde{s})s_d|}\right)=\min\{|s_d|^{-1},\min(\tilde{s})\}=\min(s).
\end{equation*}
Since $\max\{\min(\tilde{s})^{-1},|s_d|\}^{-1}=\max\{1,|s_1|,\ldots,|s_d|\}^{-1}=\min(s)$ we derive together with \eqref{eq:induktion1} the estimate \eqref{eq:induktionsvoraussetzung} for $d+1$. Hence, \eqref{eq:induktionsvoraussetzung} holds true for arbitrary dimension.
\par
With this at hand we return to arbitrary dimension $d$ and further deduce
\begin{align*}
\abs{\det A_s^i}^{-1} = \max\{1,|s_1|,\ldots,|s_{d-1}|\} \leq 1+\max\{|s_1|,\ldots,|s_{d-1}|\} \lesssim 1+\|s\|.
\end{align*}
Now we can prove the following estimate for $0<q_0<1$ and almost every $s\in\R^d$:
\begin{align*}
\int\limits_{\R^3}I_s(t)^{q_0}\,\mathrm{d}t &\lesssim \abs{\det A_s^i}^{q_0}\int\limits_{\R^d}\bigg(\sum_{i=1}^{2^{d-1}}\prod_{j=1}^d(1+\abs{(A_s^i t)_j})^{-k}\bigg)^{q_0}\,\mathrm{d}t \\
&\leq \abs{\det A_s^i}^{q_0}\sum_{i=1}^{2^{d-1}}\int\limits_{\R^d}\prod_{j=1}^d(1+|(A_s^i t)_j|)^{-kq_0}\,\mathrm{d}t \\
&\lesssim \abs{\det A_s^i}^{q_0-1}\int\limits_{\R^d}\prod_{j=1}^d(1+|t_j|)^{-kq_0}\,\mathrm{d}t \\
&\lesssim (1+\|s\|)^{1-q_0}\int\limits_{\R^d}\prod_{j=1}^d(1+|t_j|)^{-kq_0}\,\mathrm{d}t,
\end{align*}
which shows \prettyref{eq:frameinaq5}.
\par 
\underline{$I_1(s)$:} We shall now deal with the second factor in \prettyref{eq:frameinaq3} for $q_1>0$.
By \prettyref{lemma:param} \prettyref{list:paramphiphi} and the definition of $\hat\Phi$ we obtain
\begin{align*}
I_1(s)^{1/q_1} &= \int\limits_{\R^d}\abs{\hat\Phi(\omega)\hat\Phi(S_s^T\omega)}\,\mathrm{d}\omega\\
& \leq \int\limits_{\Omega_s}\abs{\omega_1}^{d-1}\biggl(\integral{\R}{\frac{\abs{\hat\Psi(\xi_1,\tilde\omega)}^2}{\abs{\xi_1}^d}}{\xi_1}\biggr)^\frac12 \biggl(\integral{\R}{\frac{\abs{\hat\Psi(\xi_1,\widetilde{S_s^T\omega})}^2}{\abs{\xi_1}^d}}{\xi_1}\biggr)^\frac12\,\mathrm{d}\omega
\end{align*}
with $\Omega_s=\{x\in\R^d:|x_1|\leq a_1,\max\{-b_i,-b_i-s_{i-1}x_1\}\leq x_i\leq\min\{b_i,b_i-s_{i-1}x_1\},i=2,\ldots d\}$.
Since $\hat\Psi$ is compactly supported and continuous, we conclude
\begin{align}\label{eq:frameinaq10}
I_1(s)^{1/q_1} \lesssim \integral{\Omega_s}{|\omega_1|^{d-1}}{\omega}.
\end{align}
In the following we assume $s>0$ componentwise, all other cases can be treated analogously by symmetry arguments.
Then, for any $\omega\in\Omega_s$ it follows from \prettyref{lemma:param} \prettyref{list:paramphiphi} that $|\omega_1|\leq 2b_is_{i-1}^{-1}$ for all $i=2,\ldots,d$, hence, \mbox{$|\omega_1|\lesssim(\max_{i=1,\ldots,d-1}s_i)^{-1}=|s|_\infty^{-1}$}.
Moreover, since $\omega\in\supp\hat\Phi$, we derive $-b_i\leq\omega_i\leq b_i$ for all $i=1,\ldots,d$.
We can now estimate \prettyref{eq:frameinaq10} in the following manner:
\begin{align*}
I_1(s)^{1/q_1} \lesssim \int\limits_{|\omega_1|\leq\min\{b_1,|s|_\infty^{-1}\}}|\omega_1|^{d-1}\,\mathrm{d}\omega_1.
\end{align*}
Assume first that $|s|_\infty^{-1}\geq b_1$, then we have
\begin{align*}
I_1(s)^{1/q_1} \lesssim \int\limits_{|\omega_1|\leq b_1}|\omega_1|^{d-1}\,\mathrm{d}\omega_1 \lesssim b_1^d \lesssim (1+\|s\|)^{-d}.
\end{align*}
On the other hand if $|s|_\infty^{-1}<b_1$ it follows that
\begin{align*}
I_1(s)^{1/q_1} \lesssim \int\limits_{|\omega_1|\leq|s|_\infty^{-1}}|\omega_1|^{d-1}\,\mathrm{d}\omega_1 \lesssim |s|_\infty^{-d} \lesssim (1+\|s\|)^{-d}.
\end{align*}
In both cases we obtain
\begin{align}\label{eq:frameinaq11}
I_1(s) \lesssim (1+\|s\|)^{-dq_1}.
\end{align}
\par
Plugging \prettyref{eq:frameinaq5} and \prettyref{eq:frameinaq11} into \prettyref{eq:frameinaq4} now yields
\begin{align}\label{eq:frameinaq12}
\int\limits_{\R^{d}}\int\limits_{\R^{d-1}}\abs{\ip{\Phi}{\psi_{(\infty,s,t)}}}^q\,\mathrm{d}s\,\mathrm{d}t &\lesssim \int\limits_{\R^{d-1}}I_0(s) I_1(s)\,\mathrm{d}s \notag\\
&\lesssim \int\limits_{\R^d}\prod_{i=1}^d(1+|t_i|)^{-kq_0}\,\mathrm{d}t\int\limits_{\R^{d-1}}(1+\|s\|)^{1-q_0-dq_1}\,\mathrm{d}s.
\end{align}
For any choice of $q_0$ we can find a $k\in\N$, such that the first integral in \prettyref{eq:frameinaq12} converges.
The second integral in \prettyref{eq:frameinaq12} is known to converge if and only if $q_0+dq_1>d$.
This can be obtained by setting $q_0=\frac{q-1}{d}$ and $q_1=\frac{d-1}{d}q+\frac{1}{d}$.
If $q>1$ this satisfies \[ q_0 + q_1 = \frac{q-1}{d} + \frac{d-1}{d} q + \frac1d = \frac1d (q-1+q(d-1)+1) = q \] and
\begin{align*}
q_0 + d q_1 &= \frac{q-1}{d} + (d-1)q + 1 = 1 + q\left(d-1+\frac1d\right)-\frac1d\\
&= d - \left( d-1+\frac1d\right) + q \left(d-1+\frac1d\right) = d + (q-1)\left( d-1+\frac1d\right) > d
\end{align*}
and we finally conclude
\begin{align*}
\int\limits_{\R^{d}}\int\limits_{\R^{d-1}}\abs{\ip{\Phi}{\psi_{(\infty,s,t)}}}^q\,\mathrm{d}s\,\mathrm{d}t < \infty.
\end{align*}
Altogether with \prettyref{eq:frameinaq1}, \prettyref{eq:frameinaq2} and \prettyref{eq:frameinaq3} we have now shown that all four summands in \prettyref{eq:frameinaq13} are bounded and this concludes the proof.
\end{proof}

At this point we intend to show that there exist functions $\hat{\Phi}$ satisfying the assumptions of \prettyref{theorem:frameinaq}.
Indeed we will show that we can find $\hat{\Psi}$ so that $\hat{\Phi}\in \mathscr{C}_0^\infty(\R^d)$.

\begin{example} \label{example:psiphi}
We fix any odd dimension $d$. Then, for $\xi=(\xi_1,\tilde\xi)$ let $\hat{\Psi}(\xi):= \hat{\psi_1}(\xi_1)\hat{\psi_2}(\tilde{\xi})$ with \[ \hat{\psi_1}(\xi_1) := \begin{cases}
\abs{\xi_1}^{\frac{d}{2}}e^{\frac{1}{(\xi_1 - 1)(\xi_1 - 3)}}, & 1 < \xi_1 < 3\\
\abs{\xi_1}^\frac{d}{2}e^{\frac{1}{(\xi_1 + 1)(\xi_1 + 3)}}, & -3 < \xi_1 < -1\\
0, &\text{otherwise}
\end{cases} \] and $\hat{\psi_2} \in \mathscr{C}^\infty_0(\R^{d-1})$ with $\hat\psi\geq 0$.
According to \prettyref{remark:choicePhi} we set
\begin{align*}
\hat\Phi(\xi) &:= \xi_1^{\frac{d-1}2}\biggl(\int\limits_{\R\setminus[-\abs{\xi_1},\abs{\xi_1}]}\frac{\abs{\hat\Psi(\omega_1,\tilde\xi)}^2}{\abs{\omega_1}^d}\,\mathrm{d}\omega_1\biggr)^{1/2} \\
&= \xi_1^\frac{d-1}{2}\abs{\hat\psi_2(\tilde\xi)}\biggl(2\int\limits_{\max \{\abs{\xi_1},1\}}^3e^{\frac{2}{(\omega_1 - 1)(\omega_1 - 3)}}\,\mathrm{d}\omega_1\biggr)^{1/2} =: \xi_1^\frac{d-1}{2}\abs{\hat\psi_2(\tilde\xi)}\hat\varphi_1(\xi_1)
\end{align*}
with $\hat\Phi(\xi)=0$ for $\abs{\xi_1}>3$.
Now we show that this function satisfies the required assumptions.
The fact that $\hat{\psi_1} \in \mathscr{C}^\infty_0(\R)$ and therefore $\hat{\Psi} \in \mathscr{C}^\infty_0(\R^d)$ is immediately obvious.
With the given construction, together with \prettyref{remark:choicePhi}, we see that the necessary condition from \prettyref{theorem:tightframe} is satisfied, i.e. the functions $\Phi$ and $\Psi$ constitute a Parseval frame.
Furthermore if we assume $\hat\Phi\in\mathscr{C}^\infty_0(\R^d)\subset\mathscr{S}(\R^d)$ then $\Phi\in\mathscr{S}(\R^d)\subset L_1(\R^d)\cap L_2(\R^d)$ and all necessary conditions on $\Phi$ are satisfied.

So we need to show that $\hat{\Phi} \in \mathscr{C}^\infty_0(\R^d)$, which means that we will show that $\hat{\varphi_1}$ is infinitely continuously differentiable since $\xi_1^\frac{d-1}{2}$ is a monomial.
To show this we need to prove that \[ \lim\limits_{x\nearrow 3} \frac{\mathrm{d}^n}{\mathrm{d}x^n} (\hat{\varphi_1}(x)) = 0 \] and \[ \lim\limits_{x\searrow 1} \frac{\mathrm{d}^n}{\mathrm{d}x^n} (\hat{\varphi_1}(x)) = 0 \] for all $n\in \N$.
Since both statements are proven in an analogous manner, we will only show the proof of the first statement and for the remainder of this example we assume $2 < x < 3$.
Since we have $\hat{\varphi_1}(x) = (f\circ g)(x)$ with $f(x) = \sqrt{x}$ and \[ g(x) = 2 \int\limits_{x}^3 e^{\frac{2}{(\omega - 1)(\omega - 3)}}\leb{\omega}, \]
we can use Fa\`a di Bruno's formula to get a closed expression for the n-th derivative.
Recall that for two functions $f$ and $g$ the identity
\begin{equation}\label{eq:faadibruno}
\frac{\mathrm{d}^n}{\mathrm{d}x^n}\bigl((f\circ g)(x)\bigr) = \sum\limits_{k=1}^n \frac{\mathrm{d}^k f}{\mathrm{d}x^k}(g(x)) B_{n,k}\Bigl(\frac{\mathrm{d}g}{\mathrm{d}x}(x),\frac{\mathrm{d}^2 g}{\mathrm{d}x^2}(x),\ldots,\frac{\mathrm{d}^{(n-k+1)} g}{\mathrm{d}x^{(n-k+1)}}(x)\Bigr)
\end{equation}
holds with $B_{n,k}$ being the Bell polynomials, i.e. \[ B_{n,k}(x_1,x_2,\ldots,x_{(n-k+1)}) = \sum \frac{n!}{j_1!\cdots j_{(n-k+1)}!} \Bigl(\frac{x_1}{1!}\Bigr)^{j_1}\cdots \Bigl(\frac{x_{(n-k+1)}}{(n-k+1)!}\Bigr)^{j_{(n-k+1)}}. \]
The sum in the above expression is taken over all $(j_1,\ldots,j_{(n-k+1)})$ with $j_1+\cdots +j_{(n-k+1)} = k$ and $j_1 + 2j_2 + \cdots + (n-k+1) j_{(n-k+1)} = n$.
The derivatives of the square root satisfy
\begin{equation*}
\frac{\mathrm{d}^k f}{\mathrm{d}x^k}(x) = c_k x^{-k+\frac12}
\end{equation*}
with $c_k$ being some constant and since because of $1<x<3$ we have
\begin{equation}\label{eq:derivg}
\frac{\mathrm{d}g}{\mathrm{d}x}(x) = -2 e^{\frac{2}{(x-1)(x-3)}}
\end{equation}
this means that for all $k \in \N$ the derivatives of $g$ satisfy \[ \frac{\mathrm{d}^k g}{\mathrm{d}x^k}(x) = Q_k(x) e^{\frac{2}{(x-1)(x-3)}} \] with $Q_k$ being some rational function without singularities in the interval $(1,3)$.
Thus, using~\prettyref{eq:faadibruno} we now have
\begin{align*}
\frac{\mathrm{d}^n \hat{\varphi_1}}{\mathrm{d}x^n}(x) &= \sum\limits_{k=1}^n c_k \bigl(g(x)\bigr)^{-k+\frac12} \sum\limits_{(j_1,\ldots,j_{(n-k+1)})} c_{n,k,j} \Bigl( Q_1(x)e^{\frac{2}{(x-1)(x-3)}} \Bigr)^{j_1} \cdots\\
&\hspace{5.5cm}\cdots \Bigl( Q_{(n-k+1)}(x)e^{\frac{2}{(x-1)(x-3)}} \Bigr)^{j_{(n-k+1)}}\\
&= \sum\limits_{k=1}^n R_{k,n}(x) \bigl(g(x)\bigr)^{-k+\frac12} \bigl( e^{\frac{2}{(x-1)(x-3)}} \bigr)^{k}\\
&= \sum\limits_{k=1}^n \biggl(\frac{\tilde{R}_{k,n}(x) \bigl(e^{\frac{2}{(x-1)(x-3)}}\bigr)^{1+\frac{1}{2k-1}}}{g(x)}\biggr)^{k-\frac12}
\end{align*}
where $R_{k,n}$ is a rational function for every $k=1,\ldots,n$ possibly changing from line to line and $\tilde{R}_{k,n}(x):= R_{k,n}(x)^{\frac{1}{k-\frac12}}$.
Since \[ \lim\limits_{x\nearrow 3} \tilde{R}_{k,n}(x) \bigl( e^{\frac{2}{(x-1)(x-3)}} \bigr)^{1+\frac{1}{2k-1}} = 0\quad\text{and}\quad \lim\limits_{x\nearrow 3} g(x) = 0 \] we use l'Hospital's rule to determine the limit of the fraction.
For the derivative of the numerator we obtain
\begin{align*}
\lefteqn{\frac{\mathrm{d}}{\mathrm{d}x}\Bigl(\tilde{R}_{k,n}(x) \bigl( e^{\frac{2}{(x-1)(x-3)}} \bigr)^{1+\frac{1}{2k-1}}\Bigr)}\\
&\hspace{1cm}= \frac{\mathrm{d}}{\mathrm{d}x} \tilde{R}_{k,n}(x) \bigl( e^{\frac{2}{(x-1)(x-3)}} \bigl)^{1+\frac{1}{2k-1}} + \tilde{R}_{k,n}(x) \frac{\mathrm{d}}{\mathrm{d}x} \bigl( e^{\frac{2}{(x-1)(x-3)}} \bigr)^{1+\frac{1}{2k-1}}\\
&\hspace{1cm}= Q(x) \bigl( e^{\frac{2}{(x-1)(x-3)}} \bigr)^{1+\frac{1}{2k-1}}
\end{align*}
where $Q$ is of the form $Q(x) = Q_2(x) (Q_1(x))^{\frac{-2k+3}{2k-1}} + Q_3(x)(Q_1(x))^{\frac{2}{2k-1}}$ with $Q_1,Q_2,Q_3$ being rational functions.
This, together with~\prettyref{eq:derivg}, yields
\begin{equation*}
\lim\limits_{x\nearrow 3} \frac{\frac{\mathrm{d}}{\mathrm{d}x}\Bigl( \tilde{R}_{k,n}(x) \bigl( e^{\frac{2}{(x-1)(x-3)}} \bigr)^{1+\frac{1}{2k-1}} \Bigr)}{\frac{\mathrm{d}}{\mathrm{d}x}\bigl(g(x)\bigr)} = \lim\limits_{x\nearrow 3} Q(x) e^{\frac{2}{(2k-1)((x-1)(x-3))}} = 0.
\end{equation*}
Thus, with l'Hospital's rule we get
\begin{align*}
\lim\limits_{x\nearrow 3} \frac{\mathrm{d}^n \hat{\varphi_1}}{\mathrm{d}x^n}(x) &= \lim\limits_{x\nearrow 3} \sum\limits_{k=1}^n \biggl(\frac{\tilde{R}_{k,n}(x) \bigl(e^{\frac{2}{(x-1)(x-3)}}\bigr)^{1+\frac{1}{2k-1}}}{g(x)}\biggr)^{k-\frac12}\\
&= \sum\limits_{k=1}^n \biggl( \lim\limits_{x\nearrow 3} \frac{\tilde{R}_{k,n}(x) \bigl(e^{\frac{2}{(x-1)(x-3)}}\bigr)^{1+\frac{1}{2k-1}}}{g(x)}\biggr)^{k-\frac12}\\
&= \sum\limits_{k=1}^n \biggl( \lim\limits_{x\nearrow 3} \frac{\frac{\mathrm{d}}{\mathrm{d}x}\Bigl(\tilde{R}_{k,n}(x) \bigl(e^{\frac{2}{(x-1)(x-3)}}\bigr)^{1+\frac{1}{2k-1}}\Bigr)}{\frac{\mathrm{d}}{\mathrm{d}x}\bigl(g(x)\bigr)}\biggr)^{k-\frac12} = 0.
\end{align*}
This proves that $\hat{\varphi_1} \in \mathscr{C}^{\infty}_0(\R)$ and therefore that $\hat{\Phi} \in \mathscr{C}^{\infty}_0(\R^d)$.
\end{example}

\subsection{Inhomogeneous shearlet coorbit spaces}\label{sec:inhomogeneousspaces}

Now we are able to give a definition of the coorbit spaces associated to our inhomogeneous shearlet frame with respect to the weighted Lebesgue spaces $L_{p,v_r}(X,\mu)$.

\begin{definition}\label{definition:shearletcoorbitspace}
Let the shearlet frame $\frF$ be chosen so that it satisfies the conditions in \prettyref{theorem:frameinaq}.
Then for $1\leq p<\infty$ and $1<\tau\leq 2$ with $p<\tau'$ the shearlet coorbit space with respect to the Lebesgue space $L_{p,v_{r}}(X,\mu)$ is defined as \[ \mathcal{SC}^{r}_{\frF,\tau,p} := \mathrm{Co}_{\frF,\tau}(L_{p,v_r}(X,\mu)) = \{ f \in (\mathcal{H}_{\tau,v_{r}})^\sim: \mathcal{SH}_{\frF,\tau} f \in L_{p,v_{r}}(X,\mu) \}. \]
It is endowed with the natural norm \[ \norm{f}{\mathcal{SC}^{r}_{\frF,\tau,p}} := \norm{\mathcal{SH}_{\frF,\tau} f}{L_{p,v_{r}}(X,\mu)}. \]
\end{definition}

These spaces are well-defined Banach spaces, which is implied by \prettyref{theorem:frameinaq}.

\begin{theorem}\label{theorem:banachspaces}
With the same assumptions as in \prettyref{theorem:frameinaq} the spaces $\mathcal{SC}^{r}_{\frF,\tau,p}$ are well-defined Banach spaces.
\end{theorem}

\begin{proof}
As stated in \prettyref{remark:assumption_fulfilled}, \prettyref{theorem:frameinaq} and \prettyref{lemma:kernel_property} imply that the assumption in \prettyref{proposition:co_properties} is fulfilled. Hence, the assertion follows.
\end{proof}

The following results are straightforward.

\begin{lemma}\label{lemma:coorbit_embeddings}
Let $1< p<q<\infty$, $1<\tau\leq 2$ with $p,q<\tau'$ and $0\leq r<s$. Furthermore let $\frF$ and $\frG$ satisfiy the conditions in \prettyref{theorem:frameinaq} with $G(\frF,\frG)\in\cA_{1,m_{v_r}}$. Then,
\begin{thmlist}
\item $\mathcal{SC}^{r}_{\frF,\tau,p}\subset\mathcal{SC}^{r}_{\frF,\tau,q}$,
\item $\mathcal{SC}^{s}_{\frF,\tau,p}\subset\mathcal{SC}^{r}_{\frF,\tau,p}$,
\item $\mathcal{SC}^{r}_{\frF,\tau,p}=\mathcal{SC}^{r}_{\frG,\tau,p}$.
\end{thmlist}
\end{lemma}

\begin{proof}
(i) and (ii) follow from \prettyref{lemma:embeddings} (ii), (iii) is a consequence of \prettyref{proposition:equality_frf}.
\end{proof}

Even though we introduced new integrability conditions on the kernel to obtain new spaces, these spaces are in fact one and the same, as the following proposition shows.

\begin{proposition}\label{proposition:coorbit_equality}
Let $1\leq p<\infty$, $1<\sigma,\tau\leq 2$ with $p<\sigma',\tau'$. Then, $\mathcal{SC}^{r}_{\frF,\tau,p}=\mathcal{SC}^{r}_{\frF,\sigma,p}$.
\end{proposition}

\begin{proof}
Assume $f\in\mathcal{SC}^{r}_{\frF,\sigma,p}$, i.e. $f\in(\mathcal{H}_{\sigma,v_r})^\sim$ with $\mathcal{SH}_{\frF,\sigma}f\in L_{p,v_r}$, by the reproducing identity and \prettyref{lemma:kernel_property} it holds $\mathcal{SH}_{\frF,\sigma}f=R_\frF(\mathcal{SH}_{\frF,\sigma}f)\in R_\frF(L_{p,v_r})\subset L_{\tau',v_r}\subset L_{\tau',\frac{1}{v_r}}$. Thus, \prettyref{lemma:equivalent_norm} yields $f\in(\mathcal{H}_{\tau,v_r})^\sim$ and $f\in\mathcal{SC}^{r}_{\frF,\tau,p}$. Equivalently the converse is shown.
\end{proof}

\begin{remark}\label{remark:lastremark}
With \prettyref{proposition:coorbit_equality} at hand the coorbit spaces solely depend on $p$ and not on $\tau$. Thus it is justified to omit the parameter $\tau$ and simply write
\begin{equation*}
\mathcal{SC}^r_{\frF,p} = \{f\in(\cH_{\tau,v_r})^\sim:\mathcal{SH}_{\frF,\tau}f\in L_{p,v_r}(X,\mu)\}
\end{equation*}
for $1\leq p<\infty$ and some $\tau$ fulfilling $p<\tau'<\infty$.
\end{remark}

\appendix

\section{}

In this appendix we will briefly discuss Young's inequality, the three-way Young's inequality and Schur's test mentioned in \prettyref{sec:coorbit_theory}.

\begin{lemma}[Young's inequality]\label{lemma:youngs_inequality}
Let $a,b\geq 0$ and $p,q>0$ with $1/p+1/q=1$, then
\begin{align*}
ab \leq \frac{a^p}{p}+\frac{b^q}{q}.
\end{align*}
\end{lemma}

\begin{lemma}[Three-way Young's inequality]\label{lemma:threeway_youngs_inequality}
Let $a,b,c\geq 0$ and $p,q,r>0$ with $1/p+1/q+1/r=1$, then
\begin{align*}
abc \leq \frac{a^p}{p}+\frac{b^q}{q}+\frac{c^r}{r}.
\end{align*}
\end{lemma}

\begin{proof}
By applying Young's inequality twice and observing $\frac{p'}{q}+\frac{p'}{r}=1$ with $\frac{1}{p}+\frac{1}{p'}=1$ we obtain
\begin{align*}
abc \leq \frac{a^p}{p} + \frac{b^{p'}c^{p'}}{p'} \leq \frac{a^p}{p} + \frac{1}{p'}\left(\frac{(b^{p'})^{q/p'}}{q/p'} + \frac{(c^{p'})^{r/p'}}{r/p'}\right) = \frac{a^p}{p}+\frac{b^q}{q}+\frac{c^r}{r},
\end{align*}
which proves the claim.
\end{proof}

\begin{lemma}[Schur's test]\label{lemma:schurs_test}
For a kernel $K:X\times X\to\C$ with $K\in\cA_{1,m_v}$ the corresponding kernel operator fulfills \[\norm{K}{L_{p,v}\to L_{p,v}}\leq\norm{K}{\cA_{1,m_v}}\] for all $1\leq p\leq\infty$.
\end{lemma}

\begin{proof}
For $p<\infty$ assume $f\in L_{p,v}$ with $\norm{f}{L_{p,v}}\leq 1$, then
\begin{align*}
\norm{K(f)}{L_{p,v}} &= \sup_{\substack{g\in L_{p',\frac{1}{v}} \\ \norm{g}{L_{p',\frac{1}{v}}}\leq 1}}\ip{K(f)}{g} \\
&\leq \sup_{\substack{g\in L_{p',\frac{1}{v}} \\ \norm{g}{L_{p',\frac{1}{v}}}\leq 1}}\int_X\int_X\abs{K(x,y)f(y)g(x)}\,\mathrm{d}\mu(x)\,\mathrm{d}\mu(y),
\end{align*}
where $p'$ denotes the H\"older-dual of $p$. By Young's inequality we obtain
\begin{align*}
&\int_X\int_X\abs{K(x,y)f(y)g(x)}\,\mathrm{d}\mu(x)\,\mathrm{d}\mu(y) \\
&\hspace{2cm} \leq \frac{1}{p}\int_X\int_X \abs{K(x,y)}m_v(x,y)\cdot\abs{f(y)}^p v(y)^p\,\mathrm{d}\mu(x)\,\mathrm{d}\mu(y) \\
&\hspace{3cm}+ \frac{1}{p'}\int_X\int_X \abs{K(x,y)}m_v(x,y)\cdot\abs{g(x)}^{p'}\frac{1}{v(x)^{p'}}\,\mathrm{d}\mu(x)\,\mathrm{d}\mu(y) \\
&\hspace{2cm} \leq \frac{1}{p}\norm{K}{\cA_{1,m_v}}\cdot\norm{f}{L_{p,v}}^p+\frac{1}{p'}\norm{K}{\cA_{1,m_v}}\cdot\norm{g}{L_{p',\frac{1}{v}}}^{p'}.
\end{align*}
Thus, $\norm{K(f)}{L_{p,v}\to L_{p,v}}\leq\norm{K}{\cA_{1,m_v}}$.
\par 
On the other hand for $p=\infty$ and $f\in L_{\infty,v}$ we have
\begin{align*}
\norm{K(f)}{L_{\infty,v}} &\leq \esssup_{x\in X}\int_X\abs{K(x,y)}m_v(x,y)\cdot\abs{f(y)}v(y)\,\mathrm{d}\mu(y)\\
&\leq \norm{K}{\cA_{1,m_v}}\cdot\norm{f}{L_{\infty,v}},
\end{align*}
which concludes the proof.
\end{proof}

\bibliographystyle{plain}
\bibliography{literature}

\begin{thebibliography}{10}

\bibitem{Curvelets}
E.J. Cand\`{e}s and D.L. Donoho.
\newblock Curvelets - a surprisingly effective nonadaptive representation for
  objects with edges.
\newblock In {Schumaker, L.L.} et~al, editor, {\em Curves and Surfaces}.
  Vanderbilt University Press, Nashville, TN, 1999.

\bibitem{Ridgelets}
E.J. Cand\`{e}s and D.L. Donoho.
\newblock Ridgelets: a key to higher-dimensional intermittency?
\newblock {\em Phil. Trans. R. Soc. Lond. A.}, 357:2495--2509, 1999.

\bibitem{Dahlke:TracesEmbeddings}
S.~Dahlke, S.~H{\"a}user, G.~Steidl, and G.~Teschke.
\newblock Shearlet coorbit spaces: traces and embeddings in higher dimensions.
\newblock {\em Monatsh. Math.}, 169:15--32, 2012.

\bibitem{ShearletCoorbits}
S.~Dahlke, G.~Kutyniok, P.~Maass, C.~Sagiv, H.-G. Stark, and G.~Teschke.
\newblock The uncertainty principle associated with the continuous shearlet
  transform.
\newblock {\em Int. J. Wavelets Multiresolut. Inf. Process.}, 6:157--181, 2008.

\bibitem{Dahlke:ShearletBanachFrames}
S.~Dahlke, G.~Kutyniok, G.~Steidl, and G.~Teschke.
\newblock Shearlet coorbit spaces and associated {B}anach frames.
\newblock {\em Appl. Comput. Harmon. Anal.}, 27(2):195--214, 2009.

\bibitem{Dahlke:ShearletArbitrary}
S.~Dahlke, G.~Steidl, and G.~Teschke.
\newblock The continuous shearlet transform in arbitrary space dimensions.
\newblock {\em J. Fourier Anal. Appl.}, 16(3):340--364, 2010.

\bibitem{Dahlke:CompactlySupp}
S.~Dahlke, G.~Steidl, and G.~Teschke.
\newblock Shearlet coorbit spaces: compactly supported analyzing shearlets,
  traces and embeddings.
\newblock {\em J. Fourier Anal. Appl.}, 17(6):1232--1255, 2011.

\bibitem{Contourlets}
M.N. Do and M.~Vetterli.
\newblock The contourlet transform: an efficient directional multiresolution
  image representation.
\newblock {\em IEEE Transactions on Image Processing}, 14(12):2091--2106, 2005.

\bibitem{FG:Coorbit1}
H.G. Feichtinger and K.~Gr\"ochenig.
\newblock A unified approach to atomic decompositios via integrable group
  representations.
\newblock In {\em Proc. Conf. ``Function Spaces and Applications''}, volume
  1302 of {\em Lecture Notes in Math.}, pages 52--73, 1988.

\bibitem{FG:Coorbit2}
H.G. Feichtinger and K.~Gr\"ochenig.
\newblock Banach spaces related to integrable group representations and their
  atomic decompositions {I}.
\newblock {\em J. Funct. Anal.}, 86:307--340, 1989.

\bibitem{FG:Coorbit3}
H.G. Feichtinger and K.~Gr\"ochenig.
\newblock Banach spaces related to integrable group representations and their
  atomic decompositions {II}.
\newblock {\em Monatsh. Math.}, 108:129--148, 1989.

\bibitem{FR:InhCoorbit}
M.~Fornasier and H.~Rauhut.
\newblock Continuous frames, function spaces, and the discretization problem.
\newblock {\em J. Fourier Anal. Appl.}, 11(3):245--287, 2005.

\bibitem{Groechenig}
K.~Gr\"ochenig.
\newblock {\em Foundations of time-frequency analysis}.
\newblock Applied and Numerical Harmonic Analysis. Birkh\"auser Boston, 2001.

\bibitem{GrohsWavefrontSet}
P.~Grohs.
\newblock Continuous shearlet frames and resolution of the wavefront set.
\newblock {\em Monatshefte f\"ur Mathematik}, 164(4):393--426, 2011.

\bibitem{Shearlets}
K.~Guo, G.~Kutyniok, and D.~Labate.
\newblock Sparse multidimensional representations using anisotropic dilation
  and shear operators.
\newblock In G.~Chen and M.J. Lai, editors, {\em Wavelets and Splines, Athens,
  GA, 2005}, pages 189--201. Nashboro Press, Nashville, TN, 2006.

\bibitem{ShearletsOptimallySparse}
K.~Guo and D.~Labate.
\newblock Optimally sparse multidimensional representation using shearlets.
\newblock {\em SIAM J. Math Anal.}, (39):298--318, 2007.

\bibitem{KempkaSchaeferUllrich:Coorbits}
H.~{Kempka}, M.~{Sch{\"a}fer}, and T.~{Ullrich}.
\newblock {General coorbit space theory for quasi-Banach spaces and
  inhomogeneous function spaces with variable smoothness and integrability}.
\newblock {\em preprint}, 2015.
\newblock arXiv:1506.07346 [math.FA].

\bibitem{ShearletsWavefrontSet}
G.~Kutyniok and D.~Labate.
\newblock Resolution of the wavefront set using continuous shearlets.
\newblock {\em Trans. Amer. Math. Soc.}, (361):2719--2754, 2009.

\bibitem{Labate:ShearletSmoothness}
D.~Labate, L.~Mantovani, and P.~Negi.
\newblock Shearlet smoothness spaces.
\newblock {\em J. Fourier Anal. Appl.}, 19(3):577--611, 2013.

\bibitem{RauhutUllrich}
H.~Rauhut and T.~Ullrich.
\newblock Generalized coorbit space theory and inhomogeneous function spaces of
  {B}esov-{L}izorkin-{T}riebel type.
\newblock {\em J. Funct. Anal.}, 260(11):3299 -- 3362, 2011.

\bibitem{Vera1}
D.~Vera.
\newblock Triebel-{L}izorkin spaces and shearlets on the cone in
  $\mathds{R}^2$.
\newblock {\em Appl. Comput. Harmon. Anal.}, 35:130--150, 2013.

\bibitem{Vera2}
D.~Vera.
\newblock Shear anisotropic inhomogeneous {B}esov spaces in $\mathds{R}^d$.
\newblock {\em Int. J. Wavelets Multiresolut. Inf. Process.}, 12(1):1450007,
  2014.

\end{thebibliography}

\end{document}